\documentclass[11pt]{article}

\usepackage{a4wide}
\usepackage{bbold,dsfont}
\usepackage{amsmath, amsthm, amssymb}
\usepackage{hyperref,natbib}
\usepackage[draft]{todonotes}
\usepackage{enumerate}
\usepackage{color}

\newcounter{rot}

\newcommand{\ignore}[1]{}

\newcommand{\BigO}{\mathcal O}
\newcommand{\bigO}{\BigO}

\newcommand{\ink}{\operatorname{ink}}

\newcommand{\indic}[1]{\mathds{1}_{\left\{#1\right\}}}

\theoremstyle{plain}
\newtheorem{theorem}{Theorem}
\newtheorem{defn}[theorem]{Definition}
\newtheorem{lemma}[theorem]{Lemma}
\newtheorem{corollary}[theorem]{Corollary}
\newtheorem{proposition}[theorem]{Proposition}
\newtheorem{remark}[theorem]{Remark}

\newtheorem*{claim*}{Claim}

\title{Mixing of the symmetric beta-binomial splitting process on arbitrary~graphs}

\author{Richard Pymar\thanks{School of Computing and Mathematical Sciences, Birkbeck, University of London, London, WC1E
7HX, UK. {\tt r.pymar@bbk.ac.uk}}
\and Nicol\'as Rivera\thanks{Facultad de Ciencias, Universidad de Valpara\'iso, Valpara\'iso, 2360102, Chile.  {\tt nicolas.rivera@uv.cl}\\
N.R.\! was supported by ANID FONDECYT 3210805 and ANID SIA 85220033}
}

\begin{document}
\maketitle
\begin{abstract}We study the mixing time of the symmetric beta-binomial splitting process on finite weighted connected graphs $G=(V,E,\{r_e\}_{e\in E})$ with vertex set $V$, edge set $E$ and positive edge-weights $r_e>0$ for $e\in E$. This is an interacting particle system with a fixed number of particles that updates through vertex-pairwise interactions which redistribute particles. We show that the mixing time of this process can be upper-bounded in terms of the maximal expected meeting time of two independent random walks on $G$. Our techniques involve using a process similar to the chameleon process invented by \cite{morris} to bound the mixing time of the exclusion process.  \end{abstract}

\section{Introduction}
In the field of econophysics, interacting particle systems have been widely used to analyse the dynamics of wealth held by agents within a network, providing insights into the distribution and flow of money within the system~\citep{yakovenko2009colloquium}. These are typically characterised by pairwise interactions between agents (represented by vertices in a graph) resulting in a redistribution of the wealth they hold (represented by particles on the vertices). 

One class of such systems which has found applications in econophysics are reshuffling models in which  each agent in an interacting pair  receives a random fraction of the total wealth they hold. In the uniform reshuffling model introduced in~\citet{dragulescu2000statistical} and discussed rigorously in~\citet{lanchier2018rigorous}, the random fraction is chosen uniformly. 

In this paper, we introduce and analyse the mixing time of the symmetric beta-binomial splitting process: a continuous-time interacting particle system on a finite connected (weighted) graph with a conservation property. Informally, the process updates by choosing randomly an edge from the graph, and redistributing the particles on the vertices of the edge according to a beta-binomial distribution. This process is a generalisation of the uniform reshuffling model, is a discrete-space version of a Gibbs sampler considered in~\citet{10.1214/20-AOP1428} and is related to the binomial splitting process of~\citet{QS} (sometimes called the binomial reshuffling model \citep{cao2022binomial}), and the KMP model of energy transport \citep{kipnis1982heat}.

Our focus is to provide general upper bounds on the mixing time of the symmetric beta-binomial splitting process on any connected graph. We achieve this through use of a chameleon process, a process which so far has only been used to bound the mixing time of exclusion processes~\citep{CP,HP,morris,olive}. We demonstrate how a chameleon process can be used more generally to understand how systems of interacting particles mix; in particular we establish a connection between the maximal expected meeting time of two independent random walks and the mixing time of the beta-binomial splitting process. Despite giving the same name to this auxiliary process, our version of the chameleon process is substantially different from those used previously; in particular it is engineered to deal with multiple particles occupying a single vertex (an event which cannot happen in the exclusion process).

As is typical with proofs that use a chameleon process, the results we obtain are not optimal in the sense that the multiplicative constants appearing in the statements are not optimized. On the other hand, the strength of this approach is in allowing us to prove results for arbitrary graphs with arbitrary edge weights.

\subsection{Model and main result}
The $m$-particle symmetric beta-binomial splitting process with parameter $s>0$ on a finite connected graph $G=(V,E,(r_e)_{e\in E})$ (with vertex set $V$, edge set $E$ and $(r_e)_{e\in E}$ a collection of positive edge weights) is the continuous-time Markov process $(\xi_t)_{t\ge0}$ on state space
\[
\Omega_{G,m}:=\left\{\xi\in\mathds{N}_0^V:\,\sum_{v\in V}\xi(v)=m\right\},
\]
with infinitesimal generator
\[
\mathcal{L}^{\mathrm{BB}(G,s,m)}f=\sum_{\{v,w\}\in E}\frac{r_{\{v,w\}}}{\sum_{e\in E}r_e}\left(\mathcal{P}_{\{v,w\}}^{\mathrm{BB}(G,s,m)}-\mathds{1}\right)f,\qquad f:\Omega_{G,m}\to\mathds{R},
\]
where, for $\xi\in\Omega_{G,m}$, $\mathcal{P}_{\{v,w\}}^{\mathrm{BB}(G,s,m)}f(\xi):=\mathds{E}[f(\xi'_{\{v,w\}})]$, and $\xi'_{\{v,w\}}$ is the random variable defined as
\[
\xi'_{\{v,w\}}(u):=\begin{cases}
X&\mbox{ if }u=v\\
\xi(v)+\xi(w)-X&\mbox{ if }u=w\\
\xi(u)&\mbox{ otherwise,}
\end{cases}
\]
with $X\sim\mbox{BetaBin}(\xi(u)+\xi(v),s,s)$.

We recover the uniform reshuffling model by setting $s=1$. We remark that in the binomial splitting process of~\cite{QS}, the random variable $X$ is chosen instead according to a binomial distribution (recall we obtain a binomial with probability parameter 1/2 by sending $s\to\infty$ in the above beta-binomial).

The symmetric \textbf{b}eta-\textbf{b}inomial \textbf{s}plitting \textbf{p}rocess (BBSP) on a connected graph with positive edge weights is irreducible on $\Omega_{G,m}$ and, by checking detailed balance, one can determine that the $m$-particle BBSP on $G$ with parameter $s$ (denoted BB$(G,s,m)$) has unique equilibrium distribution
\begin{align}\label{e:eqmbb}
\pi^{\mathrm{BB}(G,s,m)}(\xi)\propto \prod_{v\in V}\frac{\Gamma(s+\xi(v))}{\xi(v)!},\qquad\xi\in\Omega_{G,m}.
\end{align}

Recall that the total variation distance between two probability measures $\mu$ and $\nu$ defined on the same finite set $\Omega$ is
\[\|\mu-\nu\|_\mathrm{TV}:=\sum_{\omega\in\Omega}(\mu(\omega)-\nu(\omega))_+,\]
where for $x\in\mathds{R}$, $x_+:=\max\{x,0\}$. For any irreducible Markov process $(\xi_t)_{t\ge0}$ with state space $\Omega$, and equilibrium distribution $\pi$, the $\varepsilon$-total variation mixing time is \[t_\text{mix}(\varepsilon):=\inf\big\{t\ge0:\,\max_{\xi_0\in\Omega}\|\mathcal{L}[\xi_t]-\pi\|_\text{TV}\le \varepsilon\big\}\] for any $\varepsilon\in(0,1)$.

We write $t_\mathrm{mix}^{\mathrm{BB}(G,s,m)}(\varepsilon)$ for the $\varepsilon$-total variation mixing time of BB$(G,s,m)$. For $i$ and $j$ distinct vertices of $G$, we also write $\hat M_{i,j}(G)$ for the meeting time of two independent random walks started from vertices $i$ and $j$, each moving as $\mathrm{BB}(G,s,1)$, that is, the time that the two walks are on neighbouring vertices and the edge between them rings for one of the walks. Recalling that BetaBin($1,s,s)\sim\mathrm{Bernoulli}(1/2)$, we see that $\hat M_{i,j}(G)$ does not depend on $s$ and is just the meeting time of two independent random walks on the graph obtained from $G$ by halving the edge weights.

We assume throughout that $V=\{1,\ldots,n\}$. Our main result is as follows.

\begin{theorem}[Symmetric beta-binomial splitting process mixing time bound]\label{T:betabin}
Fix $s\in\mathds{Q}$ positive.
 There exists a constant $C(s)>0$ such that for any size $n$ connected graph $G$ with positive edge weights, and any integer $m\ge2$, 
\[
\forall\,\varepsilon\in(0,1/4),\qquad t_\mathrm{mix}^{\mathrm{BB}(G,s,m)}(\varepsilon)\le C(s)\log\left(\frac{n+m}{\varepsilon}\right)\max_{i,j}\mathds{E}\hat M_{i,j}(G).
\]
\end{theorem}
Our methodology does not allow us to immediately deduce results in the case of $s$ irrational. 
\begin{remark}\label{R:C}
 For $s=b/a$ with $a$ and $b$ coprime, the constant $C(s)$ can be taken to be $C'a(p^*)^{-2}\log(12a(p^*)^{-2}) \log(a+b)$, for some universal constant $C'>0$, where $p^*=(5/12)^{2s}/(6B(s,s))$ for $s<20$, and $p^*=\frac16(1-\frac{20}{s+1})$ for $s\ge20$, with $B(\cdot,\cdot)$ the beta function.
Observe that $p^*\to\frac16$ as $s\to\infty$, whereas $p^*\to0$ as $s\to0$. The quantity $1/s$ can be seen as measuring the strength that particles tend to ``clump together'', with the strength increasing as $1/s\to\infty$. Thus it is not surprising to obtain an upper-bound which increases as $s\to0$, as breaking apart clumps of particles takes longer.
\end{remark}

We recall (see, for example, \citet[Corollary 14.7]{aldousfill}) that  $\max_{i,j}\mathds{E}\hat M_{i,j}(G)\lesssim \tau_0 $, where $\tau_0$ is the average hitting time, defined as $\tau_0=\sum_{i,j\in V}\pi_i\pi_j\mathds{E}_iT_j$ with $\pi_i$  the equilibrium distribution at vertex $i$ of a simple random walk and $\mathds{E}_i T_j$  the expected hitting time of vertex $j$ by a simple random walk started from vertex $i$. \citet[Section 5.2]{aldousfill} provides a table with orders of magnitude of $\tau_0$ for certain graphs when $r_e\equiv 1$ (note that for regular graphs we must multiply the values displayed there by $|E|$ to fit within our framework). For instance, on the cycle or line, $\tau_0=\Theta(n^3)$. 

To complement Theorem~\ref{T:betabin}, we demonstrate a lower bound on the mixing time for the line $L_n:=[1,n]\cap\mathds{N}$, which is of the same order (in $n$) for fixed $s$ and $\varepsilon$, when $\log m=\Theta(\log n)$.

\begin{proposition}[Mixing time lower bound for $L_n$]\label{P:linelower}
Let $L_n$ denote the line graph on $n$ vertices with edge weights of 1, i.e.\! $r_e\equiv 1$. For any $\varepsilon\in(0,1)$, there exists a constant $C_\varepsilon>0$ such that for all integers $n,m\ge2$ and $s>0$,
\[
t_\mathrm{mix}^{\mathrm{BB}(L_n,s,m)}(\varepsilon)\ge \frac{n^3}{\pi^2}\left(\log n-\log \left(1+\frac{n}{m}+\frac1{s}\right)-C_\varepsilon\right).
\]
\end{proposition}

\subsection{Related work}
The beta-binomial splitting process is closely related to the binomial splitting process (although our methods do not obviously extend to this model). In \cite{QS}, the authors show that the binomial splitting process (as well as a more general version in which vertices have weights) exhibits total variation cutoff (abrupt convergence to equilibrium) at time $\frac12t_\mathrm{rel}\log m$ (with $t_\mathrm{rel}$ the relaxation time) for graphs satisfying a finite-dimensional geometry assumption provided the number of particles $m$ is at most order $n^2$ (they also obtain a pre-cutoff result without this restriction on particle numbers). For instance on the cycle their results show that the binomial splitting process mixes at time $\Theta(n^2\log m)$ for $m\le n^2$. On the other hand, for the beta-binomial splitting process on the cycle, our results give an upper bound of $\bigO(n^2\log(n+m))$ (with the implicit constant depending on the parameter $s$).
The beta-binomial splitting process has, in a certain sense, more dependency between the movement of the particles compared with the binomial splitting process, which in turn means any analysis on the mixing time is more involved. To see this, consider that in the binomial splitting process, when an edge rings each particle on the edge decides which vertex to jump to independently of the other particles; this independence is not present in the beta-binomial splitting process.

There has been a flurry of activity in recent years analysing mixing times of continuous mass (rather than discrete particles) redistribution processes \citep{10.1214/20-AOP1473,caputo2022spectral,pillai2018mixing,smith2013analysis}. The uniform reshuffling model (when run on the complete graph) is the discrete-space version of a Gibbs sampler on the $n$-simplex, the mixing time of which is analysed in~\citet[Example 13.3]{aldousfill} and~\cite{ASmith1}.   In~\cite{aldousfill}, the total variation mixing time of the Gibbs sampler is shown to be $\bigO(n^2\log n)$; the argument can be  used (as noted by~\cite{ASmith1}) to obtain a mixing time of $\bigO(n^2\log n)$ of the uniform reshuffling model on the complete graph (in which edge weights are all $1/(n-1)$), provided the number of particles $m$ is at least $n^{5.5}$. The arguments in~\cite{ASmith1} improve this result when $m>n^{18.5}$, obtaining $\bigO(n\log n)$ as the mixing time of the uniform reshuffling model on the complete graph in this regime. Our results improve the best known bound on the mixing time of the uniform reshuffling model on the complete graph to $\bigO(n^2\log n)$ for $m\le n^{5.5}$.

More generally, the symmetric beta-binomial splitting process is a discrete-space version of a Gibbs sampler on the $n$-simplex, in which mass is redistributed across the vertices of a ringing edge according to a symmetric beta random variable. In~\cite{10.1214/20-AOP1428}, cutoff is demonstrated at time $\frac{1}{\pi^2} n^2\log n$ for this model on the line, provided the beta parameter (which we denote by $s$ here) is at least 1. While our upper-bound for the discrete-space model holds also for some $s\in(0,1)$, we are restricted to $s\in\mathds{Q}$ by the nature of our analysis. The proof of our lower bound (Proposition~\ref{P:linelower}) for the line follows closely the analogous argument in~\cite{10.1214/20-AOP1428}. 

The beta-binomial splitting process is also related to the KMP model \citep{kipnis1982heat}  of energy transport along a chain of oscillators, and its generalisation \citep{carinci2013duality}. As described in \citet{frassek2020duality}, the dual of the generalised KMP process is identical to the beta-binomial splitting process except for the existence of additional vertices which are absorbing for the particles.

A continuous-space version of the binomial splitting process is the averaging process (also known as the repeated average model), introduced by \citet{aldous2011finite,aldouslanoue}. In this model, when an interaction occurs between two vertices, their mass is redistributed equally between them. Mixing times for this process have been studied with total variation cutoff demonstrated on the complete graph \citep{chatterjee2022phase}, and on the hypercube and complete bipartite graphs \citep{caputo2022cutoff}. A general lower bound for the mixing time of the averaging process on any connected graph is obtained by \cite{movassagh2022repeated}.

Lastly, a model similar in flavour to the beta-binomal splitting process and which also has applications in econophysics is the immediate exchange process proposed in~\cite{heinsalu2014kinetic} and its generalisation~\citep{van2016duality}. In the discrete version of the generalised immediate exchange process, when an edge updates, each vertex on the edge gives to the other vertex a random number of its particles, chosen according to a beta-binomial distribution. Again, however, our methods do not obviously extend to this model (for our methodology it is important that updates are distributionally symmetric over the vertices on a ringing edge), and obtaining bounds on the mixing time of this process appears to be an open problem.

\subsection{Road map}
To motivate the use of several auxiliary processes, we present a road map for the upper bound argument, highlighting key propositions that combine to prove Theorem~\ref{T:betabin}.

Firstly, in order to bound the total variation (TV) distance between the time-$t$ states of two BB$(G,s,m)$ processes started from arbitrary configurations, we use the triangle inequality to reduce the problem to bounding the TV distance between the time-$t$ states of two BB$(G,s,m)$ configurations which start from \emph{adjacent} configurations, that is, configurations which differ by the action of moving a single particle (from any vertex to any other). 

Formally, we say that two BB$(G,s,m)$ configurations  $\zeta^1$ and $\zeta^2$ are adjacent and write $\zeta^1\sim\zeta^2$ if there exist vertices $v$ and $w$ such that for all $y\notin\{v,w\}$, $\zeta^1(y)=\zeta^2(y)$ and $|\zeta^1(v)-\zeta^2(v)|=|\zeta^1(w)-\zeta^2(w)|=1$, i.e.\! by moving just a single particle we can obtain $\zeta^2$ from $\zeta^1$. 

\begin{proposition}
Let $\zeta_t$ and $\zeta'_t$ be two realisations of BB$(G,s,m)$ initialised at $\zeta$ and $\zeta'$ respectively. There exists a sequence of configurations $\zeta^0\sim\zeta^1\sim\cdots\sim\zeta^r$ with $r\le m$ such that
\begin{align}\label{e:1}
\|\mathcal{L}(\zeta_t)-\mathcal{L}(\zeta'_t)\|_{\mathrm{TV}}\le \sum_{i=1}^r\|\mathcal{L}(\zeta^{i-1}_t)-\mathcal{L}(\zeta^i_t)\|_{\mathrm{TV}},
\end{align}
where for each $0\le i\le r$, $\zeta^i_t$ is a realisation of BB$(G,s,m)$ started from configuration $\zeta^i$.
\end{proposition}
\begin{proof}
We choose the sequence of BB$(G,s,m)$ configurations $\{\zeta^i\}_{i=0}^r$ to satisfy
\[
\zeta=\zeta^0\sim\zeta^1\sim\cdots\sim\zeta^r=\zeta'.
\]
The result then follows by the triangle inequality for total variation.
\end{proof}
To bound the right-hand side of \eqref{e:1}, we introduce a new process which is similar to a BB$(G,s,m)$ process but has one particle marked to distinguish it from the others. 

\begin{proposition}\label{P:mabbexists}
There exists a continuous-time Markov process $(\xi_t,y_t)_{t\ge0}$ with state space $\Omega'_{G,m}:=\Omega_{G,m-1}\times V$  with the property that if we remove the marking so that all particles are identical, the process becomes BB$(G,s,m)$, that is, $(\xi_t+\delta_{y_t})_{t\ge0}$ is a realisation of BB$(G,s,m)$, where $\delta_v$ is a unit vector with value 1 in co-ordinate $v\in V$.
\end{proposition}
We shall explicitly construct a process called a MaBB (marked beta-binomial splitting) process which satisfies the requirements in Proposition~\ref{P:mabbexists} in Section~\ref{S:MaBB}.

To bound $\|\mathcal{L}(\zeta^{i-1}_t)-\mathcal{L}(\zeta^i_t)\|_{\mathrm{TV}}$, suppose that $\zeta^{i-1}$ and $\zeta^i$ differ on vertices $v$ and $w$ with $\zeta^{i-1}(v)-\zeta^i(v)=1$. Define a BB$(G,s,m-1)$ configuration $\xi$ to be $\xi(y):=\zeta^{i-1}(y)-\delta_v(y)=\zeta^i(y)-\delta_w(y)$ for all $y\in V$. As BB$(G,s,m)$ is a projection of MaBB (Proposition~\ref{P:mabbexists}), we have by the triangle inequality
\begin{align}\label{e:2}\|\mathcal{L}(\zeta^{i-1}_t)-\mathcal{L}(\zeta^i_t)\|_{\mathrm{TV}}\le \|\mathcal{L}((\xi_t,m_t))-\mathcal{L}((\xi'_t,m'_t))\|_{\mathrm{TV}},\end{align}where $(\xi_t,m_t)_{t\ge0}$ is a realisation of MaBB initialised at $(\xi,v)$ and $(\xi'_t,m'_t)_{t\ge0}$ is a realisation of MaBB initialised at $(\xi,w)$.  Next, define $\tilde m_t$ to be a random variable which, given $\xi_t$, has law $\pi_{\xi_t}$, and similarly $\tilde m'_t$ to have law $\pi_{\xi'_t}$ given $\xi'_t$. Since $\mathcal{L}((\xi_t,\tilde m_t))=\mathcal{L}((\xi'_t,\tilde m'_t))$, we use the triangle inequality again to deduce
\begin{align}\label{e:3}
\|\mathcal{L}((\xi_t,m_t))-\mathcal{L}((\xi'_t,m'_t))\|_{\mathrm{TV}}\le \|\mathcal{L}((\xi_t,m_t))-\mathcal{L}((\xi_t,\tilde m_t))\|_{\mathrm{TV}}+\|\mathcal{L}((\xi'_t,m'_t))-\mathcal{L}((\xi'_t,\tilde m'_t))\|_{\mathrm{TV}}.
\end{align}

The last stage is proving the following proposition, the proof of which uses a chameleon process adapted to this setting.
\begin{proposition}\label{P:cham}
Let $(\xi_t,m_t)$ denote the time-$t$ configuration of a MaBB initialised at $(\xi,v)\in\Omega'_{G,m}$. There exist positive constants $K_0$ and $c=c(s)$ such that for any $t>0$,
\begin{align}\label{e:5}
\|\mathcal{L}((\xi_t,m_t))-\mathcal{L}((\xi_t,\tilde m_t))\|_\mathrm{TV}\le K_0  e^{-ct/\max_{i,j}\mathds{E}\hat M_{i,j}(G)}\sqrt{a(m-1)+bn}.
\end{align}
Moreover, we can take $c(s)=(4K\log(12a(p^*)^{-2}))^{-1}$  with $K=8a/(p^*)^2$.
\end{proposition}
\begin{proof}[Proof of Theorem~\ref{T:betabin}]
Combining~\eqref{e:1}--\eqref{e:5},
\[\max_{\zeta,\zeta'}\|\mathcal{L}(\zeta_t)-\mathcal{L}(\zeta'_t)\|_{\mathrm{TV}}\le 2mK_0  e^{-t/(4K\max_{i,j}\mathds{E}\hat M_{i,j}(G)\log(12a(p^*)^{-2}))}\sqrt{a(m-1)+bn}.\]
Thus we deduce that there exists a universal $C>0$ such that if \[t\ge Ca(p^*)^{-2}\log(12a(p^*)^{-2})\log((am+bn)/\varepsilon)\max_{i,j}\mathds{E}\hat M_{i,j}(G),\] then the total variation distance between $\mathcal{L}(\zeta_t)$ and $\mathcal{L}(\zeta_t')$ is at most $\varepsilon$ for any initial configurations $\zeta$ and $\zeta'$, so the statement of Theorem~\ref{T:betabin} holds.
\end{proof}
\

\subsection{Outline of the rest of the paper}
The rest of the paper is structured as follows.
In Section~\ref{S:BBprops} we identify five key properties enjoyed by the BBSP, which includes writing the equilibrium distribution~\eqref{e:eqmbb} explicitly in terms of $a$ and $b$ (the coprime integers from Remark~\ref{R:C}). In Section~\ref{S:MaBB}, we give the construction of the MaBB process; firstly we present the dynamics of a single step, and then we show how the MaBB can be constructed `graphically'. 

The chameleon process is constructed in Section~\ref{S:cham}. We again give the dynamics of a single step, before showing how the same graphical construction can be used to build the entire trajectory of the chameleon process. Properties of the chameleon process, which allow us to make the connection to the MaBB and, ultimately, prove Proposition~\ref{P:cham}, are presented in Sections~\ref{S:champrops} and~\ref{s:loss}.
%

We complete the proof of Theorem~\ref{T:betabin} in Section~\ref{S:proof}. We present the proof of Proposition~\ref{P:linelower} in Section~\ref{S:lower}. An appendix follows, in which we collect some of the proofs requiring lengthy case analyses. Finally, we give a possible simulation of the chameleon process over three time steps to illuminate the reader further on its evolution.

\section{Key properties of the beta-binomial splitting process}\label{S:BBprops}
We fix $s\in\mathds{Q}$ positive (with $s=b/a$ for $a$ and $b$ coprime), connected graph $G$ of size $n\in\mathds{N}$, and integer $m\ge2$, and demonstrate five properties of BB($G,s,m$) needed to prove Theorem~\ref{T:betabin}.

For $e\in E$ and $\xi,\xi'\in\Omega_{G,m}$, we denote by $\mathrm{P}_e^{\mathrm{BB}(G,s,m)}(\xi,\xi')$ the probability that, given  the BB($G,s,m)$ configuration is $\xi$ and edge $e$ rings, the new configuration is $\xi'$. Further, for $v\in V$, we also write $C_{\xi,v}$ for the BB$(G,s,m+1)$ configuration which satisfies $C_{\xi,v}(u)=\xi(u)+\delta_v(u)$, for $u\in V$.
\begin{proposition}\label{P:bbprop} $BB(G,s,m)$ satisfies the following properties:
\begin{enumerate}[A.]
\item\label{a:state} BB$(G,s,m)$ is irreducible on $\Omega_{G,m}$.
\item\label{a:eqm} BB$(G,s,m)$ is reversible with equilibrium distribution 
\begin{align}\label{e:equil}
\pi^{\mathrm{BB}(G,s,m)}(\xi)\propto\prod_{\substack{v\in V:\\\xi(v)>0}}\frac1{\xi(v)!}\prod_{i=0}^{\xi(v)-1}(ai+b),\qquad\xi\in\Omega_{G,m}.
\end{align}

\item\label{a:sym} Updates are symmetric: if the configuration of BB$(G,s,m)$ is $\xi$ and edge $e=\{v,w\}$ rings to give new configuration~$\xi'$, then $\xi'(v)\stackrel{d}{=}\xi'(w)$.
\item \label{a:prob} Updates have a chance to be near even split: There exists probability $p^*\in(0,1/3)$ such that
\begin{itemize}
\item if the configuration  of BB$(G,s,m)$ is $\xi$ with $\xi(v)+\xi(w)\ge 2$ and edge $e=\{v,w\}$ rings, with probability at least $p^*$, the new configuration $\xi'$ has \[\xi'(v)\in\left[\frac13(\xi(v)+\xi(w)),\frac23(\xi(v)+\xi(w))\right],\]
\item if the configuration  of BB$(G,s,m)$ is $\xi$ with $\xi(v)+\xi(w)=2$ and edge $e=\{v,w\}$ rings,  the probability that both particles will be on the same vertex in the new configuration is at least $2p^*$.
\end{itemize} 
Moreover, it suffices to take $p^*=(5/12)^{2s}/(6B(s,s))$ for $s<20$ and $p^*=\frac16(1-\frac{20}{s+1})$ for $s\ge20$.
\item\label{a:MaBBindep} The heat kernel satisfies the following identity:  for any $\xi,\xi'\in\Omega_{G,m}$, $e=\{v,w\}\in~E$,
\begin{align*}
&(\xi'(v)+1)\mathrm{P}_e^{\mathrm{BB}(G,s,m+1)}(C_{\xi,v},C_{\xi',v})+(\xi'(w)+1)\mathrm{P}_e^{\mathrm{BB}(G,s,m+1)}(C_{\xi,v},C_{\xi',w})\\&=(\xi(v)+\xi(w)+1)\mathrm{P}_e^{\mathrm{BB}(G,s,m)}(\xi,\xi').
\end{align*}
\end{enumerate}
\end{proposition}
We defer the proof to Appendix~\ref{S:bbprop}.

\section{An auxiliary process: MaBB}\label{S:MaBB}
\subsection{Initial MaBB construction}

Recall that the MaBB must be constructed to satisfy the conditions of Proposition~\ref{P:mabbexists}. 
In addition to this, we shall also require that, given a particular edge $e$ rings, the law which governs the movement of the non-marked particles does not depend on the location of the marked particle
 (this will ensure that the uniform random variables in element 3 of the graphical construction given in Section~\ref{S:graphical} can be taken to be independent). This is not to say that the locations of the non-marked particles are independent of the location of the marked -- indeed they are not -- as the trajectory of the marked particle depends on the trajectories of the non-marked particles.

The MaBB is coupled to the BBSP so that it updates at the same times. When an edge rings in the BBSP, if the marked particle is absent from the vertices of the ringing edge, the update of the MaBB is as in the BBSP. If instead the marked particle is on one of the vertices of the ringing edge, we first remove the marked particle, then move the remaining (i.e.\! non-marked) particles as in the BBSP, and then add the marked particle back to one of the two vertices on the ringing edge with a certain law. Specifically, if $e=\{v,w\}$ is the ringing edge and the MaBB configuration before the update is $(\xi,v)$ and after the update the non-marked particles are in configuration $\xi'$, we place the marked particle on $v$ with probability
\[
\mathrm{P}_{e,\xi,\xi'}(v,v):=\frac{\xi'(v)+1}{\xi(v)+\xi(w)+1}\frac{\mathrm{P}_e^{\mathrm{BB}(G,s,m)}(C_{\xi,v},C_{\xi',v})}{\mathrm{P}_e^{\mathrm{BB}(G,s,m-1)}(\xi,\xi')},
\]and place it on $w$ with probability
\[
\mathrm{P}_{e,\xi,\xi'}(v,w):=\frac{\xi'(w)+1}{\xi(v)+\xi(w)+1}\frac{\mathrm{P}_e^{\mathrm{BB}(G,s,m)}(C_{\xi,v},C_{\xi',w})}{\mathrm{P}_e^{\mathrm{BB}(G,s,m-1)}(\xi,\xi')}.
\]This exhausts all possibilities (i.e.\! $\mathrm{P}_{e,\xi,\xi'}(v,v)+\mathrm{P}_{e,\xi,\xi'}(v,w)=1$) by  Property~\ref{a:MaBBindep}. Further, it is immediate from this construction that the movement of non-marked particles does not depend on the location of the marked particle. 

With this construction, we show that Proposition~\ref{P:mabbexists} holds.
\begin{proof}[Proof of Proposition~\ref{P:mabbexists}]
Recall $\Omega_{G,m}'$ denotes the set of configurations of the MaBB, and members of $\Omega_{G,m}'$ are of the form $(\xi,y)$ where $\xi\in\Omega_{G,m-1}$ with $\xi(v)$ denoting the number of non-marked particles at vertex $v$, and $y\in V$ denotes the location of the marked particle. 

Let $\mathrm{P}_e^{\mathrm{MaBB}}((\xi,v),(\xi',w))$ denote the probability that, given the MaBB configuration is $(\xi,v)$ and edge $e$ rings, the new configuration is $(\xi',w)$. Then in order to ensure that if we forget the marking in the MaBB we obtain the BBSP, it suffices that, for every edge $e=\{v,w\}$ and $\xi,\xi'\in\Omega_{G,m-1}$,
\begin{align}\label{e:contraction}
\mathrm{P}_e^\mathrm{MaBB}((\xi,v),(\xi',v))+\mathrm{P}_e^\mathrm{MaBB}((\xi,v),(\zeta,w))=\mathrm{P}_e^{\mathrm{BB}(G,s,m)}(C_{\xi,v},C_{\xi',v})
\end{align}
where $\zeta\in\Omega_{G,m-1}$ satisfies $\zeta(y)=\xi'(y)+\delta_v(y)-\delta_w(y)$ for $y\in V$. The reason is that if we forget the marking in either of MaBB configurations $(\xi',v)$ or $(\zeta,w)$, we obtain the same BBSP configuration $C_{\xi',v}$, and these are the only configurations with this property which are obtainable from $(\xi,v)$ when $e$ rings.

We see that~\eqref{e:contraction} holds as follows: 
\begin{align*}
&\mathrm{P}_e^\mathrm{MaBB}((\xi,v),(\xi',v))+\mathrm{P}_e^\mathrm{MaBB}((\xi,v),(\zeta,w))\\&=\mathrm{P}_{e,\xi,\xi'}(v,v)\mathrm{P}_e^{\mathrm{BB}(G,s,m-1)}(\xi,\xi')+\mathrm{P}_{e,\xi,\zeta}(v,w)\mathrm{P}_e^{\mathrm{BB}(G,s,m-1)}(\xi,\zeta)\\
&=\frac{\xi'(v)+1}{\xi(v)+\xi(w)+1}\mathrm{P}_e^{\mathrm{BB}(G,s,m)}(C_{\xi,v},C_{\xi',v})+\frac{\zeta(w)+1}{\xi(v)+\xi(w)+1}\mathrm{P}_e^{\mathrm{BB}(G,s,m)}(C_{\xi,v},C_{\zeta,w})\\
&=\frac{\xi'(v)+1}{\xi(v)+\xi(w)+1}\mathrm{P}_e^{\mathrm{BB}(G,s,m)}(C_{\xi,v},C_{\xi',v})+\frac{\xi'(w)}{\xi(v)+\xi(w)+1}\mathrm{P}_e^{\mathrm{BB}(G,s,m)}(C_{\xi,v},C_{\zeta,w})\\
&=\mathrm{P}_e^{\mathrm{BB}(G,s,m)}(C_{\xi,v},C_{\xi',v}),
\end{align*}
where the last equality uses $C_{\xi',v}=C_{\zeta,w}$.
\end{proof}

This description for the MaBB is useful as it clearly demonstrates that the movement of the non-marked particles does not depend on the location of the marked particle. There is an equivalent (distributionally-speaking) description of the MaBB which is useful for proving some other properties. Note that for $y\in\{v,w\}=e$, 
\begin{align}\label{e:MaBB2BBSP}
\mathrm{P}_e^\mathrm{MaBB}((\xi,v),(\xi',y))=\frac{\xi'(y)+1}{\xi(v)+\xi(w)+1}\mathrm{P}_e^{\mathrm{BB}(G,s,m)}(C_{\xi,v},C_{\xi',y}).
\end{align}
Thus an update of the MaBB from state $(\xi,v)$ when edge $e=\{v,w\}$ rings can be obtained by first removing the marking on the marked particle (but leaving it on the vertex) to obtain the BBSP configuration $C_{\xi,v}$, then updating according to the BBSP, which gives BBSP configuration $C_{\xi',y}$ with probability $\mathrm{P}_e^{\mathrm{BB}(G,s,m)}(C_{\xi,v},C_{\xi',y})$, and then choosing a particle from edge $e$ uniformly and applying a mark to it (so the marked particle will be on $y$ with probability $\frac{\xi'(y)+1}{\xi(v)+\xi(w)+1}$). We shall use this alternative description later in the paper (see the proof of Proposition~\ref{P:lossred}).

For $k\in\mathds{N}_0$, set $\chi(k)=ak+b$ with $a$ and $b$ the coprime integers from Property~\ref{a:eqm}. We call this the colour function. The importance of the colour function becomes apparent from the following result.
\begin{lemma}\label{L:colour}Fix vertices $v$ and $w$ with $e=\{v,w\}$ an edge of the graph. For any $\xi, \xi'\in\Omega_{G,m-1}$,
\begin{align*}\chi(\xi(v))\mathrm{P}_{e,\xi,\xi'}(v,v)+\chi(\xi(w))\mathrm{P}_{e,\xi,\xi'}(w,v)=\chi(\xi'(v)).
\end{align*}
\end{lemma}
The (yet to be defined) chameleon process will allow us to track possible locations of the marked particle in the MaBB, given the location of the non-marked particles. If we run the MaBB for a long time, and then observe that the configuration of non-marked particles is $\xi$, the probability the marked particle is on vertex $v$ will be close to $\pi_\xi(v)$ (defined in \eqref{e:pixi}). If we scale $\pi_\xi(v)$ by $a(m-1)+bn$, we obtain $\chi(\xi(v))$ (see \eqref{e:condpimips}). Together with reversibility, this is essentially the reason why Lemma~\ref{L:colour} is true. Our goal in the chameleon process will be to have $\chi(\xi(v))$ red particles on vertex $v$, for all $v$, as this will signal that the marked particle is ``mixed'' (see Proposition~\ref{P:tvboundink}). In fact, the chameleon process will always have $\chi(\xi(v))$ non-black particles on $v$ (they will be either red, white, or pink), when there are $\xi(v)$ black particles on $v$.
\begin{proof}[Proof of Lemma~\ref{L:colour}]
 Reversibility of the BBSP (Property~\ref{a:eqm}) gives that for any edge $e$ and configurations $\zeta,\zeta'\in\Omega_{G,m}$, \begin{align}\label{e:piBBSP}\pi^{\mathrm{BB}(G,s,m)}(\zeta)\mathrm{P}^{\mathrm{BB}(G,s,m)}_e(\zeta,\zeta')=\pi^{\mathrm{BB}(G,s,m)}(\zeta')\mathrm{P}^{\mathrm{BB}(G,s,m)}_e(\zeta',\zeta).\end{align}
For any $v,w\in V$ and $\xi$ and $\xi'$ which satisfy $\xi(v)+\xi(w)=\xi'(v)+\xi'(w)$, we have 
\[\sum_{y\in\{v,w\}}C_{\xi,v}(y)=\sum_{y\in\{v,w\}}C_{\xi,w}(y)=\sum_{y\in\{v,w\}}C_{\xi',v}(y)=\sum_{y\in\{v,w\}}C_{\xi',w}(y)=\xi(v)+\xi(w)+1.\]
Observe that $\mathrm{P}^\mathrm{MaBB}_e((\xi,v),(\xi',w))>0$ is equivalent to $ \mathrm{P}^\mathrm{MaBB}_e((\xi',w),(\xi,v))>0$ and implies $\xi(v)+\xi(w)=\xi'(v)+\xi'(w)$. Thus using~\eqref{e:MaBB2BBSP} and~\eqref{e:piBBSP}, we have
\begin{align*}
\pi^{\mathrm{BB}(G,s,m)}(C_{\xi,v})(\xi(v)+1)\mathrm{P}^\mathrm{MaBB}_e((\xi,v),(\xi',w))&=\pi(C_{\xi',w})(\xi'(w)+1)\mathrm{P}^\mathrm{MaBB}_e((\xi',w),(\xi,v)).
\end{align*}
By similar arguments we also have
\begin{align*}
&\pi^{\mathrm{BB}(G,s,m)}(C_{\xi,w})(\xi(w)+1)\mathrm{P}^\mathrm{MaBB}_e((\xi,w),(\xi',v))\\&=\pi^{\mathrm{BB}(G,s,m)}(C_{\xi',v})(\xi'(v)+1)\mathrm{P}^\mathrm{MaBB}_e((\xi',v),(\xi,w)),\\
\intertext{and}
&\pi^{\mathrm{BB}(G,s,m)}(C_{\xi,y})(\xi(y)+1)\mathrm{P}^\mathrm{MaBB}_e((\xi,y),(\xi',y))\\&=\pi^{\mathrm{BB}(G,s,m)}(C_{\xi',y})(\xi'(y)+1)\mathrm{P}^\mathrm{MaBB}_e((\xi',y),(\xi,y)),\qquad y\in\{v,w\}.
\end{align*}
Hence the MaBB process is reversible with equilibrium distribution \[
\pi^{\mathrm{MaBB}}((\xi,v))\propto \pi^{\mathrm{BB}(G,s,m)}(C_{\xi,v})(\xi(v)+1).
\]
For each $\xi\in\Omega_{G,m-1}$, we define \begin{align}\label{e:pixi}\pi_\xi(v):=\pi^\mathrm{MaBB}((\xi,v))/\sum_y\pi^\mathrm{MaBB}((\xi,y)),\end{align} so that
\[
\pi_\xi(v)=\frac{\pi^{\mathrm{BB}(G,s,m)}(C_{\xi,v})(\xi(v)+1)}{\sum_y\pi^{\mathrm{BB}(G,s,m)}(C_{\xi,y})(\xi(y)+1)}.
\]
Property~\ref{a:eqm} gives that 
\[
\pi^{\mathrm{BB}(G,s,m)}(C_{\xi,v})\propto \frac{a\xi(v)+b}{\xi(v)+1}\prod_{\substack{w\in V:\\\xi(w)>0}}\frac1{\xi(w)!}\prod_{i=0}^{\xi(w)-1}(ai+b),
\]
and hence
\begin{align}\label{e:condpimips}
\pi_\xi(v)=\frac{a\xi(v)+b}{a(m-1)+bn}=\frac{\chi(\xi(v))}{a(m-1)+bn}.
\end{align}
It follows that to prove the lemma, it suffices to show that
\begin{align*}
\pi_\xi(v)\mathrm{P}_{e,\xi,\xi'}(v,v)+\pi_\xi(w)\mathrm{P}_{e,\xi,\xi'}(w,v)=\pi_{\xi'}(v),
\end{align*}
equivalently,
\begin{align}
\label{e:suff}
\frac{\pi^\mathrm{MaBB}((\xi,v))}{\sum_y\pi^\mathrm{MaBB}((\xi,y))}\mathrm{P}_{e,\xi,\xi'}(v,v)+\frac{\pi^\mathrm{MaBB}((\xi,w))}{\sum_y\pi^\mathrm{MaBB}((\xi,y))}\mathrm{P}_{e,\xi,\xi'}(w,v)=\frac{\pi^\mathrm{MaBB}((\xi',v))}{\sum_y\pi^\mathrm{MaBB}((\xi',y))}.
\end{align}
Note that 
\[
\mathrm{P}_{e,\xi,\xi'}(v,v)=\frac{\mathrm{P}^\mathrm{MaBB}_e((\xi,v),(\xi',v))}{\sum_{y\in \{v,w\}}\mathrm{P}^\mathrm{MaBB}_e((\xi,v),(\xi',y))}
=\frac{\mathrm{P}^\mathrm{MaBB}_e((\xi,v),(\xi',v))}{\hat{\mathrm{P}}^\mathrm{MaBB}_{e}(\xi,\xi')},\]
where we define $\hat{\mathrm{P}}^\mathrm{MaBB}_{e}(\xi,\xi'):=\sum_{y\in \{v,w\}}\mathrm{P}^\mathrm{MaBB}_e((\xi,v),(\xi',y))$ and note that this does not depend on $v$. Thus the left-hand side of~\eqref{e:suff} can be written as
\begin{align*}
&\frac{\pi^\mathrm{MaBB}((\xi,v))\mathrm{P}^\mathrm{MaBB}_e((\xi,v),(\xi',v))+
\pi^\mathrm{MaBB}((\xi,w))\mathrm{P}^\mathrm{MaBB}_e((\xi,w),(\xi',v))}{\hat{\mathrm{P}}^\mathrm{MaBB}_{e}(\xi,\xi')\sum_y\pi^\mathrm{MaBB}((\xi,y))}\\
&=\frac{\hat{\mathrm{P}}^\mathrm{MaBB}_{e}(\xi',\xi)\pi^\mathrm{MaBB}((\xi',v))}{\hat{\mathrm{P}}^\mathrm{MaBB}_{e}(\xi,\xi')\sum_y\pi^\mathrm{MaBB}((\xi,y))}\end{align*}
using the reversibility of MaBB. Thus showing~\eqref{e:suff} is equivalent to showing
\begin{align}
\hat{\mathrm{P}}^\mathrm{MaBB}_{e}(\xi',\xi)\sum_y\pi^\mathrm{MaBB}((\xi',y))=\hat{\mathrm{P}}^\mathrm{MaBB}_{e}(\xi,\xi')\sum_y\pi^\mathrm{MaBB}((\xi,y)).
\end{align} We use reversibility to show this identity:
\begin{align*}
&\hat{\mathrm{P}}^\mathrm{MaBB}_{e}(\xi',\xi)\sum_y\pi^\mathrm{MaBB}((\xi',y))\\
&=\pi^\mathrm{MaBB}((\xi',v))\hat{\mathrm{P}}^\mathrm{MaBB}_{e}(\xi',\xi)+\pi^\mathrm{MaBB}((\xi',w))\hat{\mathrm{P}}^\mathrm{MaBB}_{e}(\xi',\xi)\\&\phantom{=}+\sum_{y\notin\{v,w\}}\pi^\mathrm{MaBB}((\xi',y))\hat{\mathrm{P}}^\mathrm{MaBB}_{e}(\xi',\xi)\\
&=\pi^\mathrm{MaBB}((\xi',v))\sum_{y\in\{v,w\}}\mathrm{P}^\mathrm{MaBB}_e((\xi',v),(\xi,y))+\pi^\mathrm{MaBB}((\xi',w))\sum_{y\in\{v,w\}}\mathrm{P}^\mathrm{MaBB}_e((\xi',w),(\xi,y))\\
&\phantom{=}\,+\sum_{y\notin\{v,w\}}\pi^\mathrm{MaBB}((\xi',y))\sum_z \mathrm{P}^\mathrm{MaBB}_e((\xi',z),(\xi,y))\\
&=\sum_{y\in\{v,w\}}\pi^\mathrm{MaBB}((\xi,y))\left(\mathrm{P}^\mathrm{MaBB}_e((\xi,y),(\xi',v))+\mathrm{P}^\mathrm{MaBB}_e((\xi,y),(\xi',w))\right)\\&\phantom{=}+\sum_{y\notin\{v,w\}}\pi^\mathrm{MaBB}((\xi',y))\mathrm{P}^\mathrm{MaBB}_e((\xi',y),(\xi,y))\\
&=\pi^\mathrm{MaBB}((\xi,v))\hat{\mathrm{P}}^\mathrm{MaBB}_{e}(\xi,\xi')+\pi^\mathrm{MaBB}((\xi,w))\hat{\mathrm{P}}^\mathrm{MaBB}_{e}(\xi,\xi')\\&\phantom{=}+\sum_{y\notin\{v,w\}}\pi^\mathrm{MaBB}((\xi,y))\mathrm{P}^\mathrm{MaBB}_e((\xi,y),(\xi',y))\\
&=\pi^\mathrm{MaBB}((\xi,v))\hat{\mathrm{P}}^\mathrm{MaBB}_{e}(\xi,\xi')+\pi^\mathrm{MaBB}((\xi,w))\hat{\mathrm{P}}^\mathrm{MaBB}_{e}(\xi,\xi')\\&\phantom{=}+\sum_{y\notin\{v,w\}}\pi^\mathrm{MaBB}((\xi,y))\sum_z\mathrm{P}^\mathrm{MaBB}_e((\xi,y),(\xi',z))\\
&=\pi^\mathrm{MaBB}((\xi,v))\hat{\mathrm{P}}^\mathrm{MaBB}_{e}(\xi,\xi')+\pi^\mathrm{MaBB}((\xi,w))\hat{\mathrm{P}}^\mathrm{MaBB}_{e}(\xi,\xi')\\&\phantom{=}+\sum_{y\notin\{v,w\}}\pi^\mathrm{MaBB}((\xi,y))\hat{\mathrm{P}}^\mathrm{MaBB}_{e}(\xi,\xi')\\
&=\hat{\mathrm{P}}^\mathrm{MaBB}_{e}(\xi,\xi')\sum_y\pi^\mathrm{MaBB}((\xi,y)).\qedhere
\end{align*}
\end{proof}

\subsection{Graphical construction of the MaBB}\label{S:graphical}

We present a `graphical construction' of the MaBB, which will also be used for the chameleon process. The motivation behind this construction is that it contains all of the random elements from which one can then deterministically construct both the MaBB and the chameleon process. In particular, it allows us to construct the MaBB and the chameleon process on the same probability space.

The graphical construction is comprised of the following elements:
\begin{enumerate}
\item A Poisson process of rate $\sum_e r_e$ which gives the times $\{\tau_1,\tau_2,\ldots\}$ at which edges ring (we also set $\tau_0=0$).
\item A sequence of edges $\{e_r\}_{r\ge 1}$ so that edge $e_r$ is the edge which rings at the $r$th time $\tau_r$ of the Poisson process; for each $r\ge1$ and $e\in E$, $\mathds{P}(e_r=e)\propto r_e$.
\item For each $r\ge 1$ an independent uniform random variable $U_r^b$ on $[0,1]$ (which will be used to determine how non-marked particles in the MaBB update at time $\tau_r$ when edge $e_r$ rings), and an independent uniform random variable $U_r^c$ on $[0,1]$ (used for updating the location of the marked particle in MaBB).

\item A sequence of independent fair coin flips $\{d_\ell\}_{\ell\ge1}$ (Bernoulli$(1/2)$ random variables).
These are only used in the chameleon process.
\end{enumerate}

We now demonstrate how the graphical construction is used to build the MaBB of interest, given an initial configuration.

Fix $u\in[0,1]$, $e=\{v,w\}\in E$, and $\xi\in\Omega_{G,m-1}$. Without loss of generality, suppose $v<w$ (recall $V=[n]$) and suppose $\{\xi_1,\ldots,\xi_r\}$ are the possible configurations of the non-marked particles that can be obtained from non-marked configuration $\xi$ when edge $e$ rings. Without loss of generality suppose they are ordered so that 
\begin{align}\label{e:ordering}|\xi_i(v)-\frac12(\xi(v)+\xi(w))|\le|\xi_j(v)-\frac12(\xi(v)+\xi(w))|\quad\text{ if and only if }\quad i\le j,\end{align} 
with any ties resolved by ordering earlier the configuration which places fewer particles on $v$.

We now define two deterministic functions $\mathrm{MaBB}:[0,1]\times E\times\Omega_{G,m-1}\to\Omega_{G,m-1}$ and $\mathrm{MaBB}^*:[0,1]\times[0,1]\times E\times \Omega_{G,m-1}\times V\to V$. 

Firstly, we define $\mathrm{MaBB}(u,e,\xi)$ to be the configuration of non-marked which satisfies, for each $1\le i\le r$,
\[
\mathrm{MaBB}(u,e,\xi)=\xi_i\quad\mathrm{ if }\quad \sum_{j<i}\mathrm{P}_e^{\mathrm{BB}(G,s,m-1)}(\xi,\xi_j)<u\le \sum_{j\le i}\mathrm{P}_e^{\mathrm{BB}(G,s,m-1)}(\xi,\xi_j).
\]
When $u$ is chosen according to a uniform on $[0,1]$ this gives that MaBB$(u,e,\xi)$ has the law of the new configuration of non-marked particles (given $e$ rings and the old configuration is $\xi$), i.e.\! for a uniform $U$ on $[0,1]$,  $\mathrm{MaBB}(U,e,\xi)$ has law $\mathrm{P}_e^{\mathrm{BB}(G,s,m-1)}(\xi,\cdot)$.

By Property~\ref{a:prob}, if $\xi(v)+\xi(w)\ge 2$, then \begin{align}\label{e:ulesspimplication}u\le p^*\quad\implies\quad\mathrm{MaBB}(u,e,\xi)(v)\in\left[\frac13(\xi(v)+\xi(w)),\frac23(\xi(v)+\xi(w))\right]\end{align} (this is the reason for choosing the ordering of the new configurations as described in~\eqref{e:ordering}, and is used in the proof of Proposition~\ref{P:lossred}).

 Secondly, for $m\in V$ and $u,u'\in[0,1]$  we set 
\begin{align}\label{e:MaBBdef}
\mathrm{MaBB}^*(u,u',e,\xi,m)=\begin{cases}
m&\mbox{if } m\notin e \mbox{ or }m\in e \mbox{ and }u'<\mathrm{P}_{e,\xi,\mathrm{MaBB}(u,e,\xi)}(m,m), \\e\setminus\{m\}&\mbox{otherwise.}
\end{cases}
\end{align}
With this construction, we have that for $U'\sim\text{Unif}\,[0,1]$,
\begin{align*}
\mathds{P}(\mathrm{MaBB}^*(u,U',\{v,w\},\xi,v)=w)&=\mathrm{P}_{\{v,w\},\xi,\mathrm{MaBB}(u,\{v,w\},\xi)}(v,w),\\
\mathds{P}(\mathrm{MaBB}^*(u,U',\{v,w\},\xi,v)=v)&=\mathrm{P}_{\{v,w\},\xi,\mathrm{MaBB}(u,\{v,w\},\xi)}(v,v),
\end{align*}
and so $\mathrm{MaBB}^*(U,U',\{v,w\},\xi,v)$ has the law of the new location of the marked particle, when edge $\{v,w\}$ updates with the marked on $v$, the non-marked particles being in configuration $\xi$, and $U,\,U'$ independent uniforms on $[0,1]$.

We can now obtain a realisation of the MaBB as follows.  Suppose we initialise at state $(\xi_0,x_0)$. Given the state at time $\tau_i$, the MaBB remains constant until the next update at time $\tau_{i+1}$, at which time
\[
\xi_{\tau_{i+1}}=\mathrm{MaBB}(U_{i+1}^b,e_{i+1},\xi_{\tau_i}),\quad m_{\tau_{i+1}}=\mathrm{MaBB}^*(U_{i+1}^b,U_{i+1}^c,e_{i+1},\xi_{\tau_i},m_{\tau_i}).
\]

\section{The Chameleon process}\label{S:cham}
\subsection{Introduction to the chameleon process}
We introduce a chameleon process in order to prove Proposition~\ref{P:cham}. In the chameleon process associated with a MaBB, the non-marked particles are replaced with black particles (which are coupled to evolve identically to  the non-marked particles). The purpose of the chameleon process is to provide a way to track how quickly the marked particle in the MaBB becomes mixed. We achieve this via the existence of red particles in the chameleon process, with each additional red particle on a vertex corresponding to an increase in the probability that in the MaBB, the marked particle is on that vertex. 
It turns out that bounding how long it takes the chameleon process to reach an all-red state (where there are $a\xi(v)+b$ red particles on each vertex $v$ when the black particles are in configuration $\xi$) when we condition on this happening before reaching a no-red state (an event we call Fill) is key to proving Proposition~\ref{P:cham}. This calculation is carried out in Section~\ref{s:loss} with the proof of Proposition~\ref{P:cham} in Section~\ref{S:proof}.

As stated previously, the chameleon process will be built using the graphical construction. The chameleon process is an interacting particle system consisting of coloured particles moving on the vertices of a graph (the same graph as the MaBB). Particles can be of four colours: black, red, pink and white. 
Each vertex $v$ in the chameleon process is occupied at a given time by a certain number, $B(v)$, of black particles and $\chi(B(v))$ non-black particles (recall $\chi$ is the colour function).

Associated with each vertex is a notion of the amount of \emph{redness}, called $\ink$ (this terminology is consistent with previous works using a chameleon process). Specifically we write $\ink_t(v)$ for the number of red particles plus half the number of pink particles at vertex $v$ at time $t$ in the chameleon process. If there are $B(v)$ black particles at vertex $v$ at time $t$ then $0\le\ink_t(v)\le \chi(B(v))$, with the minimum (resp.\! maximum) attained when all non-black particles are white (resp.\! red).

We use the initial configuration of the MaBB to initialise the chameleon process. Each non-marked particle on a vertex in the MaBB configuration corresponds to a black particle at the same vertex in the chameleon process. The vertex with the marked particle in the MaBB is initialised in the chameleon process with all non-black particles as red. Every other vertex has all non-black particles as white.

The chameleon process consists of rounds of length $T$ (a parameter of the process), and at the end of some rounds is a depinking time. Whether we have a depinking time (at which we remove all pink particles, replacing them all with either red or white particles) will depend on the numbers of red, pink and white particles in the graph at the end of that round.

 If at the start of the round there are fewer red than white particles then we shall assign to each red particle a unique white particle; thus each red particle has a paired white particle. Later, our interest will be in determining how many red particles `meet' their paired white particle during a round, where two particles are said to meet if, at some moment in time, they are both on the same ringing edge (unless they start on the same vertex, they will be on different vertices when they meet). If there are fewer white than red particles at the start of the round we shall reverse roles so that each white particle gets a unique paired red particle.

In the chameleon process we can only create new pink particles (by re-colouring red and white particles) at the meeting times of paired particles. It is this restriction which will lead to us taking the round length to be the maximal expected meeting time of two random walks.

In previous works using other versions of the chameleon process, the idea of using paired particles is not used (it is not needed). It becomes useful here because \emph{a priori} there is no constant (not depending on the number of particles or size of the graph) bound on the number of particles which may occupy a vertex. As a result, without using pairing, it turns out we would need to understand the movement of 3 coloured particles simultaneously, rather than the movement of one red and one white until their meeting time.

\subsection{A single step of the chameleon process}
Our construction of the chameleon process is such that when an edge rings, we first observe how the non-marked particles move in the MaBB and move the black particles in the same way. Given the new configuration of black particles, the number of non-black particles on the vertices is determined by the colour function $\chi$. After observing the movement of the black particles, we shall then determine the movement of the red particles (and if we have to pinken any) then the pre-existing pink particles (i.e.\! not any just-created pink particles) and finally the white particles.

To specify more precisely an update, we introduce some notation. We shall define a probability \[\theta(v)=\theta(v,e,B(v),B(w),B'(v),B'(w),R(v),R(w),P(v),P(w))\] which is a function of a vertex $v$, an edge containing that vertex $e=\{v,w\}$, and non-negative integers \[B(v), B(w), B'(v), B'(w), R(v),R(w), P(v), P(w)\] which satisfy $B'(v)\le B(v)+B(w)$, $B(v)+B(w)=B'(v)+B'(w)$, $R(v)+P(v)\le \chi(B(v))$, $R(w)+P(w)\le \chi(B(w))$.  The integers $B,R,P$ shall represent the numbers of black/red/pink on the vertices of the edge $e$ just prior to its ringing, and $B'(v)$, $B'(w)$ the number of black particles on $v, w$ just after $e$ rings. 

For simplicity we write $R_{v,w}$ for $R(v)+R(w)$ and $P_{v,w}$ for $P(v)+P(w)$.

To define $\theta(v)$ we also define integers
\begin{align*}
\ell(v)=\ell(v,R_{v,w},B'(w))&:=\{R_{v,w}-\chi(B'(w))\}\vee0,\\
u(v)=u(v,R_{v,w},B'(v))&:=\chi(B'(v))\wedge R_{v,w},\\
u(w)=u(w,R_{v,w},B'(w))&:=\chi(B'(w))\wedge R_{v,w}=R_{v,w}-\ell(v),\\
\ell^P(v)=\ell^P(v,R_{v,w},P_{v,w},B'(w))&:=\{P_{v,w}-\chi(B'(w))+u(w)\}\vee0,\\
u^P(v)=u^P(v,R_{v,w},P_{v,w},B'(w))&:=\{\chi(B'(v))-u(v)\}\wedge P_{v,w}.
\end{align*}
We also set $\ell(w)=R_{v,w}-u(v)$. The idea behind these definitions is the following. The values of $\chi(B'(v))$ and $\chi(B'(w))$ impose restrictions on the number of non-black particles which can occupy vertices $v$ and $w$ after the update. For example, the number of red particles on $v$ cannot exceed $\chi(B'(v))$; this gives an upper limit of $u(v)$ for the number of red particles that we can place onto $v$ after the update. On the other hand, the number of red particles on $w$ after the update cannot exceed $\chi(B'(w))$, which in turn means that the number of red particles on $v$ has to be at least $R_{v,w}-\chi(B'(w))$, giving a lower limit of $\ell(v)$. The difference between these values, i.e.\! $u(v)-\ell(v)=R_{v,w}-\ell(v)-\ell(w)$ is the number of \emph{flexible reds}, that is, the number of red particles which can be either on $v$ or $w$ after the update. It is these flexible reds that we get a chance to pinken, with pink particles representing particles which are half red and half white. Once the values of $u(v)$ and $u(w)$ have been determined, based on how the black particles move, we can then place the pre-existing pink particles. We again have to ensure that the number of non-black particles on $v$ does not exceed $\chi(B'(v))$, and now there could be at most $u(v)$ red particles, so we restrict to placing at most $\chi(B'(v))-u(v)$ pink particles; this gives $u^P(v)$. There is a similar restriction on vertex $w$ and through this we obtain a lower bound $\ell^P(v)$ on the number of pink particles to place onto $v$. The role of $\theta(v)$ is to give the probability of placing the lower limits on $v$, with $1-\theta(v)$ then the probability of placing the upper limits on $v$. We choose $\theta(v)$ to satisfy
\begin{equation}
\begin{aligned}\label{e:thetadef}
&\theta(v)[\ell(v)+\frac12\ell^P(v)]+(1-\theta(v))[u(v)+\frac12 u^P(v)]\\&=(R(v)+\frac12 P(v))\mathrm{P}_{e,B,B'}(v,v)+(R(w)+\frac12 P(w))\mathrm{P}_{e,B,B'}(w,v)=:m^*(v).
\end{aligned}
\end{equation}
This particular choice of $\theta(v)$ is necessary to ensure that the expected amount of ink at a vertex (given numbers of black particles on the vertices) matches the probability that the marked particle in the MaBB process is on that vertex (given the location of non-marked particles), see Lemma~\ref{L:onestep}.

The following lemmas shows that such a $\theta(v)$ exists and give bounds on its value.

\begin{lemma}[Existence of $\theta(v)$]\label{L:thetaexist}
For every $e,v,w,B,B',R,P$, 
\[
\ell(v)+\frac12\ell^P(v)\le m^*(v)\le u(v)+\frac12 u^P(v),
\]
and so in particular there exists $\theta(v)\in[0,1]$ satisfying~\eqref{e:thetadef}.
\end{lemma}
\begin{lemma}[Bounds on $\theta(v)$]\label{L:thetabounds}
Fix $e=\{v,w\}$, $B$ and $\eta\in(0,1/2)$. If $B'$ satisfies
\[
\mathrm{P}_{e,B,B'}(v,v),\,\mathrm{P}_{e,B,B'}(w,w)\in[\eta,1-\eta],
\]
then $\theta(v)\in[\eta,1-\eta]$. 
\end{lemma}
The proofs of these lemmas involve lengthy (but straightforward) case analyses and can be found in Appendix~\ref{S:App}.

We now describe in full detail the dynamics of a single step of the chameleon process, including the role of $\theta(v)$. We show how this fits with the graphical construction in the next section. We assume that pairings of red and white particles have already happened (these happen at the beginning of each round, more details are provided on this in the next section on how this is achieved through ``label configurations''). 

As a preliminary step, we remove all non-black particles from the vertices of the ringing edge and place them into a pooled pile. They will be redistributed to the vertices during the steps described below. We update the black particles from $B$ to $B'$ according to the law of the movement of the non-marked particles in the MaBB (recall that the movement of non-marked particles does not depend on the location of the marked).

\underline{Step 1:} [Place lower bounds] \\
If there are no red particles on $v$ or $w$, skip straight to Step 4. Otherwise we proceed as follows. We introduce a notion of reserving paired particles in this step and put the lower bounds $\ell(v)$ and $\ell(w)$ of red particles onto vertices $v$ and $w$. In choosing red particles to use for the lower bounds, it is important to avoid as much as possible the paired red particles (i.e.\! those reds for which their paired white is also on the ringing edge) so that they can be reserved for the set of flexible reds, as only reds which are both flexible and paired can actually be pinkened. Thus, when choosing from the pooled pile for the $\ell(v)+\ell(w)$ reds for the lower bounds, we shall first choose the non-paired reds (and the specific ones chosen -- i.e.\! the vertex they started from at this update step and the label if they have one (see the next section for a discussion on when and how to label particles) -- is made uniformly). If there are insufficient non-paired reds, then once they are placed we choose from the paired reds (again uniformly).

\underline{Step 2:} [A fork in the road] \\
With probability $2[\theta(v)\wedge (1-\theta(v))]$
 proceed to Step 3a; otherwise skip Step 3a and proceed to Step 3b. 
 
 \underline{Step 3a:} [Create new pink particles] \\ 
 Let $k$ denote the number of paired red particles remaining in the pile after Step 1.  Select (uniformly) 
\begin{align}\label{e:cond}k\wedge\left\{\lceil [(|R|\wedge|W|)+|P|/2]/3\rceil-|P|/2\right\}\end{align}
paired red particles from the pile\footnote{By taking this minimum we ensure that the number of pink particles created won't exceed a certain threshold.}, where $|R|:=\sum_{v\in V}R(v)$, and similarly for $|W|$ and $|P|$ (where $W(v)$ denotes the number of white particles on $v$). These are coloured pink and placed onto $v$. The paired white particles of these selected red particles are also coloured pink and placed onto $w$. Any paired red and any non-paired red left in the pile are then each independently placed onto $v$ or $w$ equally likely. Now proceed to Step 4.

\underline{Step 3b:} [Place remaining red particles] \\If $\theta(v)<1/2$ put any remaining red particles from the pile onto $v$. As a result there will now be $u(v)$ red particles on $v$. If instead $\theta(v)\ge 1/2$, put any remaining red particles from the pile onto $w$ (and so there are $u(w)$ red particles on $w$.) Now proceed to Step 4.

\underline{Step 4:} [Place old pink particles]\\ There may be some pink particles remaining in the pool (which were already pink at the start of the update). If not, skip to Step 5; otherwise with probability $\theta(v)$, put $\ell^P(v)$ of these pink particles on $v$, and the rest (i.e.\! $u^P(w)$ of them) on $w$. With the remaining probability, instead put $u^P(v)$ of them on $v$ and the rest on $w$.

\underline{Step 5:} [Place white particles] \\The only possible particles left in the pile are white particles. These are placed onto $v$ and $w$ to ensure that the total number of non-black particles now on $v$ is $\chi(B'(v))$ (which also ensures there are $\chi(B'(w))$ non-black particles on $w$ since $\chi(B(v))+\chi(B(w))=\chi(B'(v))+\chi(B'(w))$ and no particles are created or destroyed). The choice of which white particles are put onto $v$ is done uniformly.

The next result shows the usefulness of reserving in guaranteeing a certain number of reserved pairs remain in the pool after Step 1. 

Write $R^p_{v,w}$ for the number of paired red particles on $e=\{v,w\}$, and set $R^q_{v,w}=R(v)+R(w)- R^p_{v,w}$. 

\begin{lemma}\label{L:pairres}
If there are $k$ paired red particles on ringing edge $e=\{v,w\}$ then the number that are left remaining in the pooled pile after Step 1 above is at least $k\wedge \chi(B'(v))\wedge \chi(B'(w))$. Further, on the event that $\chi(B'(v))/\chi(B'(w))\in[\gamma,1/\gamma]$ for some $\gamma\in(0,1)$, the probability any particular paired red particle remains in the pool after Step 1 is at least $\gamma$ uniformly over $B$, $R^q_{v,w}$, $R^p_{v,w}$ and $P$.
\end{lemma}
We defer the proof to Appendix~\ref{S:App}.

The next result gives the expected amount of ink after one step of using this algorithm. We state the result in terms of the first update, given any initial conditions. Recall $m^*(v)$ is defined in~\eqref{e:thetadef}.

\begin{lemma}\label{L:onestep}For any $v,w\in V$, $B,R,P$ initial configurations of black, red and pink particles, and $B'$ the configuration of black particles just after the first update (at time $\tau_1$),
\[\mathds{E}[\ink_{\tau_1}(v)\mid B, B', R, P, \{e_1=\{v,w\}\}]=m^*(v).\]
\end{lemma}
\begin{proof}
Recall that each red particle contributes 1 to the ink value of the vertex it occupies, and each pink particle contributes $1/2$. 

We first consider the contribution to $\ink_{\tau_1}(v)$ which comes from the particles placed onto $v$ in Step 1. This is straightforward: we place $\ell(v)$ particles onto $v$ from the pile and these are all red, thus the contribution to $\ink_{\tau_1}(v)$ from Step 1 is simply $\ell(v)$. 

At Step 2 we do not place any new particles onto the vertices, but we do decide whether to proceed with Step 3a or Step 3b. If we do Step 3a then each red particle (paired or otherwise) in the pool will in expectation contribute a value of 1/2 to $\ink_{\tau_1}(v)$: either it gets coloured pink as does its paired white and one of them is placed onto $v$, or it stays red and is placed onto $v$ with probability 1/2. If we do Step 3b and $\theta(v)<1/2$ then we place the remaining red particles on $v$ which gives a total of $u(v)$ red on $v$. If instead $\theta(v)\ge 1/2$, we do not place any more red particles on $v$.

Finally at Step 4 we place the pre-existing pink particles, each contributing $1/2$ to the ink of the vertex they are placed on.

Putting these observations together we obtain
\begin{align*}
&\mathds{E}[\ink_{\tau_1}(v)\mid B,B',R,P,\{e_1=\{v,w\}\}]\\&=\ell(v)+2[\theta(v)\wedge(1-\theta(v))]\frac{u(v)-\ell(v)}{2}+(1-2[\theta(v)\wedge(1-\theta(v))])\indic{\theta(v)<1/2}(u(v)-\ell(v))\\
&\phantom{=}+\theta(v)\frac{\ell^P(v)}{2}+(1-\theta(v))\frac{u^P(v)}{2}\\
&=\ell(v)+\indic{\theta(v)<1/2}\big\{\theta(v)(u(v)-\ell(v))+(1-2\theta(v))(u(v)-\ell(v))\big\}\\&\phantom{=}+\indic{\theta(v)\ge1/2}\big\{(1-\theta(v))(u(v)-\ell(v))\big\}+\theta(v)\frac{\ell^P(v)}{2}+(1-\theta(v))\frac{u^P(v)}{2}\\
&=\ell(v)+(1-\theta(v))(u(v)-\ell(v))+\theta(v)\frac{\ell^P(v)}{2}+(1-\theta(v))\frac{u^P(v)}{2}\\
&=\theta(v)\left(\ell(v)+\frac{\ell^P(v)}{2}\right)+(1-\theta(v))\left(u(v)+\frac{u^P(v)}{2}\right)\\
&=m^*(v).\qedhere
\end{align*}
\end{proof}

\subsection{The evolution of the chameleon process}
We define a ``particle configuration'' to be a function $V\to\mathds{N}_0$, which, in practice, will be the configuration of red, black, pink or white particles.  For $S$ a particle configuration we define $|S|:=\sum_{v\in V}S(v)$. We also define a ``label configuration'' to be a function $[a(m-1)+bn]\to V\cup\{0\}$, which will give the vertex occupied by the labelled particle of a certain colour (and which has value 0 if there is no particle of a given label). We discuss further this labelling now.

At the start of every round we shall pair some red particles with an equal number of white particles. The way we do this, and how we track the movement of the paired particles, is by labelling paired red and white particles with a unique number. 
Suppose there are $r$ red particles at the start of the $\ell$th round, and this is less than the number of white particles (otherwise, reverse roles of red and white in the following). We label the red particles with labels $1,\ldots, r$ such that for any pair of vertices $v$ and $w$, the label of any red particle on vertex $v$ is less than the label of any red particle on vertex $w$ if and only if $v<w$.  In other words, we label red particles on vertex $1$ first, then label red particles on $2$, and continue until we have labelled all $r$ red particles.  We similarly label $r$ white particles with the (same) rule that for any pair of vertices $v$ and $w$, the label of any white particle on vertex $v$ is less than the label of any white particle on vertex $w$ if and only if $v<w$. A labelled red particle and a labelled white particle are pairs if they have the same label. 

For every time, we will have two label configurations: one for the red particles and one for the white. Suppose $L$ is such a label configuration for the red particles at a certain time. Then the number of labelled red particles at this time is equal to \[\max\{ i: \,1\le i\le a(m-1)+bn, L(i)\neq0\},\] so in particular,  $L(i)=0$ for any $i$ larger than the number of labelled red particles.

There are several aspects of the update rule which require external randomness: in Step 1, to choose which particles make up the lower bounds, in Step 2 to determine whether we proceed with Step 3a or Step 3b, in Step 3a choosing which paired red particles to pinken and how to place the remaining red particles in the pile, in Step 4 how to place the old pink particles, and in Step 5 to place the white particles. To fit the chameleon process into the framework of the graphical construction, we shall use random variables $\{U^c_i\}_{i\ge1}$ as the source of the needed randomness with $U^c_i$ used at time $\tau_i$ (and we shall not make it explicit \emph{how} this is done). Further, and importantly, we shall do this in a way such that the randomness used at Step 1 is independent of the randomness used at Step 2 (it is standard that this is possible, see for example~\citet[Section 4.6]{williams1991probability}).

The random variables $\{U^b_i\}_{i\ge1}$ are used to determine how the black particles move so that they move in the same way as the non-marked particles in the MaBB.

For independent uniforms $U$, $U'$ on $[0,1]$, an edge $e$, particle configurations $B$ of black particles, $P$ of pink particles, and $R$ of red particles, and label configurations $L^R$ for red particles, $L^W$ for white particles, we define $\mathrm{C}(U,U',e,B,R,P,L^R,L^W)$ to be a quintuple with the first component equal to $\mathrm{MaBB}(U,e,B)$, the second (resp.\! third) component denoting the configuration of red (resp.\! pink) particles, and the fourth (resp.\! fifth) component denoting the label configuration of red (resp.\! white) particle just after edge $e$ rings if before this edge rang the configuration of black, red and pink was given by $B$, $R$ and $P$, the label configuration of red particles was $L^R$, and of white was $L^W$, and we use $U'$ as the source of randomness for Steps 1--5 as described above (in practice we shall take $U'$ to be $U^c_i$ for some $i\ge1$).

\begin{defn}[Chameleon process]\label{d:cham}
The chameleon process with round length $T>0$ and associated with a MaBB initialised at $(\xi,x)$ is the quintuple $(B_t^\mathrm{C},R_t^\mathrm{C},P_t^\mathrm{C},L^R_t,L^W_t)_{t\ge0}$ where $B_t^\mathrm{C},R_t^\mathrm{C}$ and $P_t^\mathrm{C}$ are particle configurations and $L^R_t$, $L^W_t$ are label configurations for each $t\ge0$, with the following properties:
\begin{enumerate}
\item (Initial values) $B_0^\mathrm{C}(v)=\xi(v)$, $R_0^\mathrm{C}(v)=\chi(\xi(x))\delta_{x}(v)$, and $P_0^\mathrm{C}(v)=0$, for all $v\in V$,
\[L_0^R(i)=\begin{cases}x &\mbox{for }1\le i\le N_0:=\chi(\xi(x))\wedge[a(m-1)+bn-\chi(\xi(x))],\\0&\mbox{otherwise,}\end{cases}\]
and
\[
L_0^W(i)=\begin{cases}\min\Big\{\ell\in\{1,\ldots,n\}\setminus\{x\}:\,\sum_{\substack{k\in[\ell]:\\k\neq x}}\chi(\xi(k))\ge i\Big\} &\mbox{for }1\le i\le N_0,\\0&\mbox{otherwise.}\end{cases}
\] 
\item (Updates during rounds) For each $i\ge1$,
\begin{align*}
(B_{\tau_i}^\mathrm{C},R_{\tau_i}^\mathrm{C},P_{\tau_i}^\mathrm{C},L_{\tau_i}^R,L_{\tau_i}^W)&=\mathrm{C}(U_i^b,U_i^c,e_i,B_{\tau_{i}-}^\mathrm{C},R_{\tau_{i}-}^\mathrm{C},P_{\tau_{i}-}^\mathrm{C}).
\end{align*}
\item (Particle configuration updates at end of rounds) For each $i\ge1$ such that \begin{align}\label{e:icond}\sum_{v\in V}P_{iT-}^\mathrm{C}(v)\ge\min\left\{\sum_{v\in V}R_{iT-}^\mathrm{C}(v),\sum_{v\in V}\left(\chi(B_{iT-}^\mathrm{C}(v))-R_{iT-}^\mathrm{C}(v)-P_{iT-}^\mathrm{C}(v)\right)\right\},\end{align}
we set 
\[
B_{iT}^\mathrm{C}(v)=B_{iT-}^\mathrm{C}(v),\,\, R_{iT}^\mathrm{C}(v)=R_{iT-}^\mathrm{C}(v)+d_iP_{iT-}^\mathrm{C}(v),\,\, P_{iT}^\mathrm{C}(v)=0\quad\text{for all }v\in V;
\]and if $i$ does not satisfy~\eqref{e:icond} then we set 
\[
B_{iT}^\mathrm{C}(v)=B_{iT-}^\mathrm{C}(v),\,\, R_{iT}^\mathrm{C}(v)=R_{iT-}^\mathrm{C}(v),\,\, P_{iT}^\mathrm{C}(v)=P_{iT-}^\mathrm{C}(v)\quad\text{for all }v\in V.
\]
\item (Label configuration updates at end of rounds) For each $i\ge1$ we define \[N_i:=\sum_v R_{iT}^\mathrm{C}(v)\wedge\Big[a(m-1)+bn-\sum_v \left(R_{iT}^\mathrm{C}(v)+P_{iT}^\mathrm{C}(v)\right)\Big]\] and set
\begin{align*}L_{iT}^R(j)&=\begin{cases}\min\Big\{\ell\in[n]:\,\sum_{k=1}^\ell R^\mathrm{C}_{iT}(k)\ge j\Big\}&\mbox{for }1\le j\le N_i,\\
0&\mbox{otherwise,}\end{cases}\\ L_{iT}^W(j)&=\begin{cases}\min\Big\{\ell\in[n]:\,\sum_{k=1}^\ell \left(\chi(B^\mathrm{C}_{iT}(k))-R^\mathrm{C}_{iT}(k)-P^\mathrm{C}_{iT}(k)\right)\ge j\Big\}&\mbox{for }1\le j\le N_i,\\0&\mbox{otherwise.}\end{cases}\end{align*}
\end{enumerate}
\end{defn}
We can obtain the number of white particles $W_t^\mathrm{C}(v)$ at time $t$ on a vertex $v$ using $W_t^\mathrm{C}(v)+R_t^\mathrm{C}(v)+P_t^\mathrm{C}(v)=\chi(B_t^\mathrm{C}(v))$.

We write $\mathcal{C}(m)$ for the space of possible configurations of the chameleon process in which the underlying MaBB has $m-1$ non-marked particles.

We note from this definition that the process also updates at the ends of rounds, i.e.\! at times of the form $iT$ for $i\ge1$. At these times if the number of pink particles is at least the number of red or white particles (i.e.\! if \eqref{e:icond} holds), then we have a \emph{depinking} (and call this time a \emph{depinking time}) in which all pink particles are removed from the system. To do this, we use the coin flips $d_i$ given in the graphical construction. If time $iT$ is a depinking time then we re-colour all pink particles red simultaneously if $d_i=1$, otherwise if $d_i=0$ we re-colour them all white.

A simulation of the chameleon process for the first few update times appears in Appendix~\ref{A:sim}.

\section{Properties of the chameleon process}\label{S:champrops}
\subsection{Evolution of ink}\label{SS:evol}
In this section we suppose that the chameleon process considered is associated with a MaBB initialised at $(\xi,x)$.
\begin{lemma}\label{L:inkonly}
The total ink in the system only changes at depinking times.
\end{lemma}
\begin{proof}
This is a straightforward observation as the only particles that change colour at an update time that is not a depinking are paired red and white particles. But since we colour each in the pair pink, the total ink does not change.
\end{proof}
Let $\widehat{\ink}_j$ denote the ink in the system just after the $j$th depinking time and $D_j$ the time of the $j$th depinking. The process $\{\widehat{\ink}_j\}_{j\ge1}$ evolves as a Markov chain; the following result gives its transition probabilities. This result is similar to \citet[Proposition 7.3]{olive} for the chameleon process used there.
\begin{lemma}\label{L:inkmc}
For $j\in\mathds{N}$, $\widehat{\ink}_{j+1}\in\{\widehat{\ink}_j-\Delta(\widehat{\ink}_j),\widehat{\ink}_j+\Delta(\widehat{\ink}_j)\}$ a.s., where for each $r\in\mathds{N}$,
\[
\Delta(r):=\left\lceil\frac{\min\{r,a(m-1)+bn-r\}}{3}\right\rceil.
\]
Moreover, conditionally on $\{\widehat{\ink}_\ell\}_{\ell=0}^j$, each possibility has probability $1/2$.
\end{lemma}
\begin{proof}
Fix $j\in\mathds{N}$. After each depinking is performed there are no pink left in the system, and so $\widehat{\ink}_j$ is equal to the number of red particles at time $D_j$, $|R_{D_j}^\mathrm{C}|=\sum_vR_{D_j}^\mathrm{C}(v)$. As the number of non-black particles is fixed at $a(m-1)+bn$, it follows that the number of white particles at time $D_j$ is $|W_{D_j}^\mathrm{C}|=\sum_vW_{D_j}^\mathrm{C}(v)=a(m-1)+bn-\widehat{\ink}_j$.

Observe that every time a red and white particle pair are pinkened, we lose one red and one white, and gain two pink particles.

It can be easily checked that for $p$ and $q$ positive integers with $p$ even, \[p<q\Leftrightarrow \lceil (q+p/2)/3\rceil-p/2>0.\]
In other words, while the number of pink particles remains less than the minimum of the number of red and white, the chameleon process will still create new pink particles (recall the number of pink particles created in Step 3a of the chameleon process); conversely, the chameleon process will stop producing new pink particles as soon as the number of pink particles is at least the minimum of the number of red and white particles. Moreover, once it stops producing new pink particles, the number of pink created is the smallest number which ensures that the number of pink is at least the number of red or white; we can see this by observing that 
\[
p+2\left(\lceil (q+p/2)/3\rceil-p/2\right)=2\lceil (q+p/2)/3\rceil
\]
is the smallest even integer which is at least
\[
q-\left(\lceil (q+p/2)/3\rceil-p/2\right)=(q+p/2)-\lceil (q+p/2)/3\rceil.
\]

Thus the number of pink particles created just before the next depinking time (at time $D_{j+1})$ is the smallest $p$ even satisfying $p\ge |W_{D_j}^\mathrm{C}|-p/2$ or $p\ge |R_{D_j}^\mathrm{C}|-p/2$, which is $p=2\Delta(\widehat{\ink}_j)$.

At the depinking time $D_{j+1}$, the pink particles either all become white (and $\widehat{\ink}_{j+1}=\widehat{\ink}_j-\Delta(\widehat{\ink}_j)$) or they all become red (and $\widehat{\ink}_{j+1}=\widehat{\ink}_j+\Delta(\widehat{\ink}_j)$). Which event happens depends just on the outcome of the independent fair coin flip $d_{j+1}$.
\end{proof}
\begin{lemma}\label{L:inkmart}The total ink in the system is a martingale and is absorbed in finite time in either 0 or $a(m-1)+bn$. Further, the event \[
\mathrm{Fill}:=\left\{\lim_{t\to\infty}\ink_t=a(m-1)+bn\right\}
\]
has probability $\chi(\xi(x))/(a(m-1)+bn)$.
\end{lemma}
\begin{proof}
The fact that total ink is a martingale follows from Lemma~\ref{L:inkmc} and the behaviour of the chameleon process at depinking times. The probability of event Fill then follows by the martingale property and the dominated convergence theorem (total ink is bounded by $a(m-1)+bn$), as in the proof of Lemma 7.1 of \cite{olive}.
\end{proof}
\begin{corollary}\label{cor:fill}
For $\zeta\in \Omega_{G,m-1}$ and $t\ge0$,
\[
\mathds{P}(\{B_t^\mathrm{C}=\zeta\}\cap\mathrm{Fill})=\mathds{P}(B_t^\mathrm{C}=\zeta)\mathds{P}(\mathrm{Fill})=\mathds{P}(B_t^\mathrm{C}=\zeta)\frac{\chi(\xi(x))}{a(m-1)+bn}.
\]
\end{corollary}
\begin{proof}
This follows from Lemma~\ref{L:inkmart} and the fact that event Fill only depends on the outcomes of the coin flips $\{d_i\}_i$ whereas the movement of the black particles is independent of these coin flips.
\end{proof}
\begin{lemma}\label{L:limits}
For all $t\ge0$ and $v\in V$, $\ink_t(v)\le \chi(B^\mathrm{C}_t(v))$.
\end{lemma}
\begin{proof}This follows simply from the fact that the number of non-black particles on a vertex with $B$ black particles is always $\chi(B)$. This is true at time 0, and Steps 1 to 5 guarantee this at update times which are not depinkings. Finally, at depinking times we do not change the number of particles on vertices, only their colour. Observe also that $\ink_t(v)=\chi(B^\mathrm{C}_t(v))$ if at time $t$ all non-black particles on $v$ are red.
\end{proof}

The next result shows that, during a single round and until they meet, a pair of paired red-white particles move (marginally) as independent random walks on the graph, which stay in place with probability $1/2$ when an incident edge rings. For two independent random walks $X,Y$ on a graph $G$ (each of which move by jumping from their current vertex $v$ to a neighbour $w$ when edge $\{v,w\}$ rings), we write $M^{X,Y}$ for their meeting time -- the first time they are on neighbouring vertices, and the edge between them rings for one of the walks (they each have their own independent sequence of edge-rings). If the walks start on the same vertex, we say their meeting time is 0. We let $\hat G$ denote the graph $(V,E,\{r_e/2\}_{e\in E})$, that is, we halve the rates on the edges of graph $G$. 
\begin{lemma}\label{L:beforemeet}Fix $u,v\in V$, $u\neq v$, and $i\in\mathds{N}_0$. Let $X$ and $Y$ be independent random walks on $\hat G$ with $X_0=u$, $Y_0=v$. For any $1\le j\le \sum_vR_{iT}^\mathrm{C}(v)\wedge \sum_vW_{iT}^\mathrm{C}(v)$, conditionally on $L_{iT}^R(j)=u$ and $L_{iT}^W(j)=v$,  for all $t\in[iT,iT+\{T\wedge M^{X,Y}\})$, we have
\[
(L_t^R(j),L_t^W(j))\stackrel{d}{=}(X_{t-iT},Y_{t-iT}).
\]
\end{lemma}
\begin{proof}We make use of Property~\ref{a:sym}. Suppose edge $e=\{v,w\}$ rings during time interval $[iT,(i+1)T)$ and the black particles update from configuration $B$. Suppose $B'$ is a possible configuration of the black particles as a result of the update. Let $\tilde B$ be the configuration of black particles with $\tilde B(v)=B'(w)$, $\tilde B(w)=B'(v)$ and for $z\notin e$, $\tilde B(z)=B'(z)=B(z)$. As black particles update as non-marked particles in MaBB, $B'$ and $\tilde B$ are equally likely to be the configuration of black particles after the update, by Property~\ref{a:sym}. We claim that the probability that a labelled red particle (similarly labelled white particle) will be on $v$ after the update if configuration $B'$ is chosen as the new black configuration is the same as the probability the same labelled red particle (respectively, labelled white particle) will be on $w$ if configuration $\tilde B$ is chosen. This will suffice since prior to meeting, a paired red and white particle will never be on the same ringing edge.

This claim will follow from showing that $\ell(v)=\tilde\ell(w)$, $\ell^P(v)=\tilde\ell^P(w)$, $u(v)=\tilde u(w)$, $u^P(v)=\tilde u^P(w)$ and $\theta(v)=\tilde\theta(w)$, where the notation with tilde refers to the update in which $\tilde B$ is chosen, and notation without the tilde to the update in which $B'$ is chosen. The identities regarding the lower and upper values are immediate from their definitions. To show $\theta(v)=\tilde\theta(w)$, observe that
\begin{equation}
\begin{aligned}\label{e:tildetheta}
&\tilde\theta(v)[\tilde\ell(v)+\frac12\tilde\ell^P(v)]+(1-\tilde\theta(v))[\tilde u(v)+\frac12 \tilde u^P(v)]\\&=(R(v)+\frac12 P(v))\mathrm{P}_{e,B,\tilde B}(v,v)+(R(w)+\frac12 P(w))\mathrm{P}_{e,B,\tilde B}(w,v).
\end{aligned}
\end{equation}
But by Property~\ref{a:sym}, we have
\begin{align*}
\mathrm{P}_{e,B,\tilde B}(v,v)&=\frac{\tilde B(v)+1}{B(v)+B(w)+1}\frac{\mathrm{P}_e^{\mathrm{BB}(G,s,m)}(C_{B,v},C_{\tilde B,v})}{\mathrm{P}_e^{\mathrm{BB}(G,s,m)}(B,\tilde B)}\\
&=\frac{ B'(w)+1}{B(v)+B(w)+1}\frac{\mathrm{P}_e^{\mathrm{BB}(G,s,m)}(C_{B,v},C_{B',w})}{\mathrm{P}_e^{\mathrm{BB}(G,s,m)}(B,B')}\\
&=\mathrm{P}_{e,B, B'}(v,w),
\end{align*}
and similarly $\mathrm{P}_{e,B,\tilde B}(w,v)=\mathrm{P}_{e,B,B'}(w,w)$. Plugging these into~\eqref{e:tildetheta} shows that $\tilde\theta(v)$ solves the same equation as $\theta(w)$, hence they are equal; similarly $\theta(v)=\tilde\theta(w)$.
\end{proof}

\subsection{From ink to total variation}
In this section we show a crucial connection between the MaBB initialised at $(\xi,x)$ and its associated chameleon process. To emphasise the dependence of $\ink_t$ on the initial configuration of the MaBB, we shall sometimes write it as $\ink_t^{(\xi,x)}$.

\begin{proposition}\label{P:MaBBtoC} Let $(\xi_t,m_t)$ denote the time-$t$ configuration of a MaBB initialised at $(\xi,x)\in\Omega'_{G,m}$. For every $t\ge0$ and $(\zeta,y)\in \Omega'_{G,m}$, 
\[
\mathds{P}\big((\xi_t,m_t)=(\zeta,y)\big)=\mathds{E}\left[\frac{\ink_t^{(\xi,x)}(y)}{\chi(\xi(x))}\indic{B_t^\mathrm{C}=\zeta}\right].
\]
\end{proposition}

The proof of Proposition~\ref{P:MaBBtoC}  is similar in spirit to the proof of Lemma 1 of \cite{morris}.
We introduce a new process $M^*$ which will also be constructed using the graphical construction. This process is similar to the chameleon process in that vertices are occupied by particles of various colours (black, red, pink and white). Like in the chameleon process, if there are $B$ black particles on a vertex, then there are $\chi(B)$ non-black particles.  The process $M^*$ evolves exactly as the chameleon process except we replace Step 3a with Step 3a$^\prime$, described below. Further, $M^*$ does not have any updates at the ends of rounds (so in particular no depinking times). As a result the number of red, white and pink particles remain constant over time. We use the same terminology (e.g.\! ink) for process $M^*$.

\underline{Step 3a$^\prime$:} Any  red particles left in the pile are each independently placed onto $v$ or $w$ equally likely.

It can be shown (following the same proof) that Lemma~\ref{L:onestep} holds also for $M^*$:
\begin{lemma}\label{L:onestepM*}
For any $v,w\in V$, $B,R,P$ initial configurations of black, red and pink particles, and $B'$ the configuration of black particles just after the first update (at time $\tau_1$),
\[\mathds{E}^{M^*}[\ink_{\tau_1}(v)\mid B, B', R, P, \{e_1=\{v,w\}\}]=m^*(v).\]\end{lemma}
\begin{lemma}\label{L:MipstoM*} Fix $(\xi,x)\in\Omega'_{G,m}$, random variable $\ink_0(y)$ taking values in $[0,\chi(\xi(y))]\cap(\mathds{N}_0/2)$ for each $y\in V$, and denote by $(\xi_t,m_t)$ the time-$t$ configuration of a MaBB which starts from a random configuration $(\xi_0,m_0)$ satisfying almost surely
\[
\forall\, y\in V\quad \mathds{P}(m_0=y\mid \xi_0)=\mathds{E}^{M^*}\left[\frac{\ink_0(y)}{\chi(\xi(x))}\bigm| B_0^\mathrm{C}\right],
\]
where $M^*$ starts with configuration of black particles $B_0^\mathrm{C}=\xi_0$ and with initial ink value of $\ink_0(y)$ at each $y\in V$.
Then for all $t\ge 0$, almost surely
\[
\forall\, y\in V\quad\mathds{P}(m_t=y\mid (\xi_s)_{0\le s\le t})=\mathds{E}^{M^*}\left[\frac{\ink_t(y)}{\chi(\xi(x))}\bigm| (B_s^\mathrm{C})_{0\le s\le t}\right].
\]
\end{lemma}
\begin{proof}
As $(\xi_s)_{s\ge 0}$ and $(B^\mathrm{C}_s)_{s\ge0}$ are constructed using the same $(U_r^b)_{r=1}^\infty$, they are equal almost surely. 

It suffices to show the statement at the update times. We shall use induction. The base case (time $\tau_0=0$) follows from the assumption. 
Fix $r\in\mathds{N}$ and suppose the result holds up to (and including) time $\tau_{r-1}$. 

Observe that  by the strong Markov property and Lemma~\ref{L:onestepM*}
(and recall the choice of $\theta$ from~\eqref{e:thetadef} and also that $B^\mathrm{C}_{\tau_r}=\mathrm{MaBB}(U_r^b,e_r,B^\mathrm{C}_{\tau_{r-1}})$), for any $y\in V$, almost surely
\begin{align*}
&\mathds{E}^{M^*}[\ink_{\tau_r}(y)\mid U_r^b,e_r,B^\mathrm{C}_{\tau_{r-1}},R_{\tau_{r-1}}^\mathrm{C},P_{\tau_{r-1}}^\mathrm{C}]\\&=\indic{y\in e_r}\bigg\{\left[R_{\tau_{r-1}}^\mathrm{C}(y)+\frac12 P_{\tau_{r-1}}^\mathrm{C}(y)\right]\,\mathrm{P}_{e_r,B^\mathrm{C}_{\tau_{r-1}},B^\mathrm{C}_{\tau_r}}(y,y)\\
&\phantom{=}+\left[R_{\tau_{r-1}}^\mathrm{C}(e_r\setminus\{y\})+\frac12 P_{\tau_{r-1}}^\mathrm{C}(e_r\setminus\{y\})\right]\,\mathrm{P}_{e_r,B^\mathrm{C}_{\tau_{r-1}},B^\mathrm{C}_{\tau_r}}(e_r\setminus\{y\},y)\bigg\} +\indic{y\notin e_r}\ink_{\tau_{r-1}}(y)\\
&=\indic{y\in e_r}\Big\{\ink_{\tau_{r-1}}(y)\mathrm{P}_{e_r,B^\mathrm{C}_{\tau_{r-1}},B^\mathrm{C}_{\tau_r}}(y,y)+\ink_{\tau_{r-1}}(e_r\setminus\{y\})\mathrm{P}_{e_r,B^\mathrm{C}_{\tau_{r-1}},B^\mathrm{C}_{\tau_r}}(e_r\setminus\{y\},y)\Big\}\\&\phantom{=}+\indic{y\notin e_r}\ink_{\tau_{r-1}}(y).
\end{align*}
Taking an expectation, the first half of the first term above becomes 
\begin{align*}
&\mathds{E}^{M^*}\left[\indic{y\in e_r}\ink_{\tau_{r-1}}(y)\mathrm{P}_{e_r,B^\mathrm{C}_{\tau_{r-1}},B^\mathrm{C}_{\tau_r}}(y,y)\bigm|(B^\mathrm{C}_s)_{s\le\tau_r}\right]
\\&=\mathds{E}^{M^*}\left[\mathds{E}^{M^*}\left[\indic{y\in e_r}\ink_{\tau_{r-1}}(y)\mathrm{P}_{e_r,B^\mathrm{C}_{\tau_{r-1}},B^\mathrm{C}_{\tau_r}}(y,y)\bigm|e_r,\,(B^\mathrm{C}_s)_{s\le\tau_r}\right]\Bigm|(B^\mathrm{C}_s)_{s\le\tau_r}\right]\\
&=\mathds{E}^{M^*}\left[\indic{y\in e_r}\mathrm{P}_{e_r,B^\mathrm{C}_{\tau_{r-1}},B^\mathrm{C}_{\tau_r}}(y,y)\mathds{E}^{M^*}\left[\ink_{\tau_{r-1}}(y)\bigm|(B^\mathrm{C}_s)_{s\le\tau_{r-1}}\right]\Bigm|(B^\mathrm{C}_s)_{s\le\tau_r}\right]\\
&=\mathds{E}^{M^*}\left[\chi(\xi(x))\mathds{P}(m_{\tau_{r-1}}=y\mid (\xi_s)_{s\le \tau_{r-1}})\,\indic{y\in e_r}\mathrm{P}_{e_r,B^\mathrm{C}_{\tau_{r-1}},B^\mathrm{C}_{\tau_r}}(y,y)\Bigm|(B^\mathrm{C}_s)_{s\le\tau_r}\right],
\end{align*}using in the penultimate step that almost surely \[\mathds{E}^{M^*}\left[\ink_{\tau_{r-1}}(y)\bigm|e_r, \,(B^\mathrm{C}_s)_{s\le\tau_{r}}\right]=\mathds{E}^{M^*}\left[\ink_{\tau_{r-1}}(y)\bigm|(B^\mathrm{C}_s)_{s\le\tau_{r-1}}\right],\] since $B^\mathrm{C}_{\tau_r}=\mathrm{MaBB}(U_r^b,e_r,B^\mathrm{C}_{\tau_{r-1}})$ and $\ink_{\tau_{r-1}}(y)$ is independent of $e_r$ and $U_r^b$; and using the induction hypothesis in the last step. Similarly,
\begin{align*}
&\mathds{E}^{M^*}\left[\indic{y\in e_r}\ink_{\tau_{r-1}}(e_r\setminus\{y\})\mathrm{P}_{e_r,B^\mathrm{C}_{\tau_{r-1}},B^\mathrm{C}_{\tau_r}}(e_r\setminus\{y\},y)\bigm|(B^\mathrm{C}_s)_{s\le\tau_r}\right]
\\&=\mathds{E}^{M^*}\left[\chi(\xi(x))\mathds{P}(m_{\tau_{r-1}}=e_r\setminus\{y\}\mid (\xi_s)_{s\le \tau_{r-1}},\,e_r)\,\indic{y\in e_r}\mathrm{P}_{e_r,B^\mathrm{C}_{\tau_{r-1}},B^\mathrm{C}_{\tau_r}}(e_r\setminus\{y\},y)\Bigm|(B^\mathrm{C}_s)_{s\le\tau_r}\right],
\end{align*}
and thus
\begin{align}\label{e:M*ink}
&\mathds{E}^{M^*}\left[\frac{\ink_{\tau_r}(y)}{\chi(\xi(x))}\bigm| (B^\mathrm{C}_s)_{s\le \tau_r}\right]\\\notag
&=\mathds{E}^{M^*}\Big[\mathds{P}(m_{\tau_{r-1}}=y\mid (\xi_s)_{s\le \tau_{r-1}})\left[\indic{y\in e_r}\mathrm{P}_{e_r,\xi_{\tau_{r-1}},\xi_{\tau_r}}(y,y)+\indic{y\notin e_r}\right]\\
&\phantom{=\mathds{E}\Big[}+\indic{y\in e_r}\mathds{P}(m_{\tau_{r-1}}=e_r\setminus\{y\}\mid (\xi_s)_{s\le \tau_{r-1}},\,e_r)\,\mathrm{P}_{e_r,\xi_{\tau_{r-1}},\xi_{\tau_r}}(e_r\setminus\{y\},y)\Bigm|(\xi_s)_{s\le\tau_r}\Big].\notag
\end{align}
On the other hand, using the definition of MaBB$^*$ from~\eqref{e:MaBBdef},
\begin{align*}
&\mathds{P}(m_{\tau_r}=y\mid(\xi_s)_{s\le\tau_r})\\
&=\mathds{P}(\mathrm{MaBB}^*(U_r^b,U^c_r,e_r,\xi_{\tau_{r-1}},m_{\tau_{r-1}})=y\mid(\xi_s)_{s\le\tau_r})\\
&=\mathds{P}\bigg(m_{\tau_{r-1}}\left[\indic{m_{\tau_{r-1}}\notin e_r}+\indic{m_{\tau_{r-1}}\in e_r}\indic{U_r^c<\mathrm{P}_{e_r,\xi_{\tau_{r-1}},\xi_{\tau_r}}(m_{\tau_{r-1}},m_{\tau_{r-1}})}\right]\\&\phantom{=\mathds{P}\bigg(}+(e_r\setminus\{m_{\tau_{r-1}}\})\indic{m_{\tau_{r-1}}\in e_r}\indic{U_r^c\ge\mathrm{P}_{e_r,\xi_{\tau_{r-1}},\xi_{\tau_r}}(m_{\tau_{r-1}},m_{\tau_{r-1}})}=y\Bigm|(\xi_s)_{s\le\tau_r}\bigg)\\
&=\mathds{P}\left(\{m_{\tau_{r-1}}=y\in e_r\}\cap\{U_r^c<\mathrm{P}_{e_r,\xi_{\tau_{r-1}},\xi_{\tau_r}}(y,y)\}\mid(\xi_s)_{s\le\tau_r}\right)\\
&\phantom{=}+\mathds{P}\left(\{m_{\tau_{r-1}}=e_r\setminus\{y\},\,y\in e_r\}\cap\{U_r^c<\mathrm{P}_{e_r,\xi_{\tau_{r-1}},\xi_{\tau_r}}(e_r\setminus\{y\},y)\}\mid(\xi_s)_{s\le\tau_r}\right)\\
&\phantom{=}+\mathds{P}\left(\{m_{\tau_{r-1}}=y\notin e_r\}\mid(\xi_s)_{s\le\tau_r}\right).
\end{align*}
Using the tower property of conditional expectation we condition further on $e_r$, and then use that given $e_r$ and $(\xi_s)_{s\le \tau_r}$, the event $\{U_r^c<\mathrm{P}_{e_r,\xi_{\tau_{r-1}},\xi_{\tau_r}}(y,y)\}$ is independent of the event $\{m_{\tau_{r-1}}=y\}\cap\{y\in e_r\}$, to obtain
\begin{align*}
&\mathds{P}(m_{\tau_r}=y\mid(\xi_s)_{s\le\tau_r})\\
&=\mathds{E}\left[\mathds{E}\left[\indic{m_{\tau_{r-1}}=y}\big(\indic{y\in e_r}\indic{U_r^c<\mathrm{P}_{e_r,\xi_{\tau_{r-1}},\xi_{\tau_r}}(y,y)}+\indic{y\notin e_r}\big)\mid(\xi_s)_{s\le\tau_r},  e_r\right]\Bigm| (\xi_s)_{s\le\tau_r}\right]\\
&\phantom{=}+\mathds{E}\left[\mathds{E}\left[\indic{y\in e_r}\indic{m_{\tau_{r-1}}=e_r\setminus\{y\}}\indic{U_r^c<\mathrm{P}_{e_r,\xi_{\tau_{r-1}},\xi_{\tau_r}}(e_r\setminus\{y\},y)}\mid(\xi_s)_{s\le\tau_r}, e_r\right]\Bigm| (\xi_s)_{s\le\tau_r}\right]\\\
&=\mathds{E}\left[\mathds{E}\left[\indic{m_{\tau_{r-1}}=y}\left[\indic{y\in e_r}\mathrm{P}_{e_r,\xi_{\tau_{r-1}},\xi_{\tau_r}}(y,y)+\indic{y\notin e_r}\right]\mid (\xi_s)_{s\le\tau_r},e_r\right]\Bigm| (\xi_s)_{s\le\tau_r}\right]\\
&\phantom{=}+\mathds{E}\left[\mathds{E}\left[\indic{y\in e_r}\indic{m_{\tau_{r-1}}=e_r\setminus\{y\}}\mathrm{P}_{e_r,\xi_{\tau_{r-1}},\xi_{\tau_r}}(e_r\setminus\{y\},y)\mid (\xi_s)_{s\le\tau_r},e_r\right]\Bigm| (\xi_s)_{s\le\tau_r}\right]\\
&=\mathds{E}\Big[\mathds{P}^\mathrm{MaBB}(m_{\tau_{r-1}}=y\mid (\xi_s)_{s\le \tau_{r-1}})\,\left[\indic{y\in e_r}\mathrm{P}_{e_r,\xi_{\tau_{r-1}},\xi_{\tau_r}}(y,y)+\indic{y\notin e_r}\right]\\
&\phantom{=\mathds{E}\Big[}+\indic{y\in e_r}\mathds{P}^\mathrm{MaBB}(m_{\tau_{r-1}}=e_r\setminus\{y\}\mid (\xi_s)_{s\le \tau_{r-1}},\,e_r)\,\mathrm{P}_{e_r,\xi_{\tau_{r-1}},\xi_{\tau_r}}(e_r\setminus\{y\},y)\Bigm|(\xi_s)_{s\le\tau_r}\Big].
\end{align*}
which agrees with \eqref{e:M*ink} and so completes the inductive step.
\end{proof}

We now turn to the proof of Proposition~\ref{P:MaBBtoC}.

\begin{proof}[Proof of Proposition~\ref{P:MaBBtoC}]
We shall need a list of times at which updates occur for the chameleon process; recall that the chameleon process updates at times $\{\tau_r\}_{r\ge1}$ but also at depinking times. To this end, we set $\hat\tau_0=0$ and for each $r\ge1$, we set
\[
\hat \tau_r=(\min\{\tau_m:\,\tau_m>\hat\tau_{r-1}\})\wedge(\min\{D_i:\,D_i> \hat\tau_{r-1},\,i\in\mathds{N}\}).
\]
Similarly a hat placed on notation (e.g.\! $\hat e_r$) refers to the (in this example) edge chosen at time $\hat\tau_r$. If this is a depinking time then we set $\hat e_r=V$.

Next, for each $r\ge1$ we introduce process $(M_t^r)_{t\ge0}$ which is constructed using the graphical construction. Each of these processes is a process in which vertices are occupied by particles of various colours, and we initialise them all with the initial configuration of the chameleon process. Prior to time $\hat\tau_r$, process $M^r$ evolves exactly as the chameleon process; at and after time $\hat\tau_r$ it evolves as $M^*$ (so in particular there are no more changes to the colours of particles). Note that in all these processes the black particles have the same trajectory and this matches the trajectory of the non-marked particles in the MaBB. Note also that $M^1$ is identical to $M^*$. We shall prove by induction on $r$ that for all $r\ge1$, 
\begin{align}\label{e:MaBBtoMr}
\forall\,t>0,\,y\in V\quad
\mathds{P}(m_t=y\mid\xi_t)=\mathds{E}^{M^r}\left[\frac{\ink_t^{(\xi,x)}(y)}{\chi(\xi(x))}\bigm| B^\mathrm{C}_t\right]\quad\mbox{a.s}.
\end{align}
This will prove the proposition since the chameleon process is the almost sure limit of $M^r$ as $r\to\infty$.

The case $r=1$ follows from Lemma~\ref{L:MipstoM*} since  $\ink_0^{(\xi,x)}(y)=\chi(\xi(x))$ if $y=x$ and otherwise $\ink_0^{(\xi,x)}(y)=0$ (thus the assumption of the lemma holds).

We fix $r'\in\mathds{N}_0$, assume \eqref{e:MaBBtoMr} holds for $r=r'$ and show it holds for $r=r'+1$. 
Observe that before time $\hat\tau_{r'}$, $M^{r'+1}=M^{r'}$ so for $t<\hat\tau_{r'}$, for all $y$, almost surely
\begin{align*}
\mathds{P}(m_t=y\mid\xi_t)&=\mathds{E}^{M^{r'}}\left[\frac{\ink_t^{(\xi,x)}(y)}{\chi(\xi(x))}\bigm| B^\mathrm{C}_t\right]=\mathds{E}^{M^{{r'}+1}}\left[\frac{\ink_t^{(\xi,x)}(y)}{\chi(\xi(x))}\bigm| B^\mathrm{C}_t\right].
\end{align*}
After time $\hat\tau_{r'}$, $M^{r'+1}$ evolves as $M^*$; so assuming that for all $y\in V$,
\begin{align}\label{e:mipsmr}
\mbox{a.s.}\quad\mathds{P}(m_{\hat\tau_{r'}}=y\mid\xi_{\hat\tau_{r'}})=\mathds{E}^{M^{r'+1}}\left[\frac{\ink^{(\xi,x)}_{\hat\tau_r'}(y)}{\chi(\xi(x))}\bigm| B^\mathrm{C}_{\hat\tau_{r'}}\right],
\end{align}
then by Lemma~\ref{L:MipstoM*} we have that for all $t>\hat\tau_{r'}$, for all $y\in V$,
\[
 \mbox{a.s.}\quad\mathds{P}(m_{t}=y\mid(\xi_s)_{\hat\tau_{r'}\le s\le t})=\mathds{E}^{M^{r'+1}}\left[\frac{\ink^{(\xi,x)}_{t}(y)}{\chi(\xi(x))}\bigm| (B^\mathrm{C}_{s})_{\hat\tau_{r'}\le s\le t}\right].
\]
The inductive step is then complete by taking an expectation and using that black particles have the same trajectory as the non-marked, almost surely. Thus it remains to prove \eqref{e:mipsmr}. We fix $y\in V$ and decompose according to three events, which partition the probability space:
\begin{itemize}
\item $E_1:=\bigcup_{i\ge1}\{y\notin \hat e_{r'}\}\cap\{\hat\tau_{r'}=\tau_i\}$ (the update is not a depinking time and $y$ is not on the ringing edge)
\item $E_2:=\bigcup_{i\ge1}\{y\in \hat e_{r'}\}\cap\{\hat\tau_{r'}=\tau_i\}$ (the update is not a depinking time but $y$ is on the ringing edge)
\item $E_3:=\bigcup_{i\ge1}\{\hat\tau_r=D_i\}$ (the update  is a depinking time) 
\end{itemize}
On event $E_1$, as $y$ is not on a ringing edge at time $\hat\tau_{r'}$, the value of $\ink_t^{(\xi,x)}(y)$ does not change at time $\hat\tau_{r'}$ in either of the processes $M^{r'}$ or $M^{r'+1}$; since they agree prior to this time, we deduce that almost surely
\begin{align}\label{e:E1}
\mathds{E}^{M^{r'+1}}\left[\frac{\ink^{(\xi,x)}_{\hat\tau_{r'}}(y)}{\chi(\xi(x))}\indic{E_1}\bigm| B^\mathrm{C}_{\hat\tau_{r'}}\right]=\mathds{E}^{M^{r'}}\left[\frac{\ink^{(\xi,x)}_{\hat\tau_{r'}}(y)}{\chi(\xi(x))}\indic{E_1}\bigm| B^\mathrm{C}_{\hat\tau_{r'}}\right].
\end{align}
On event $E_2$, we may pinken some particles at time $\hat\tau_{r'}$ in process $M_{r'+1}$. Nevertheless, by Lemmas~\ref{L:onestep} and~\ref{L:onestepM*} (and again since the processes agree prior to this time), we see that their expected ink values agree, i.e.
\begin{align}\label{e:E2}
\mathds{E}^{M^{r'+1}}\left[\frac{\ink_{\hat\tau_{r'}}(y)}{\chi(\xi(x))}\indic{E_2}\bigm| B^\mathrm{C}_{\hat\tau_{r'}}\right]=\mathds{E}^{M^{r'}}\left[\frac{\ink_{\hat\tau_{r'}}(y)}{\chi(\xi(x))}\indic{E_2}\bigm| B^\mathrm{C}_{\hat\tau_{r'}}\right].
\end{align}
Finally, on event $E_3$, $M^{r'}$ does not update. On the other hand, almost surely
\begin{align}
&\mathds{E}^{M^{r'+1}}\left[\frac{\ink_{\hat\tau_{r'}}(y)}{\chi(\xi(x))}\indic{E_3}\bigm| B^\mathrm{C}_{\hat\tau_{r'}}\right]\notag\\&=\sum_{i=1}^\infty\mathds{E}^{M^{r'+1}}\left[\frac{\ink_{\hat\tau_{r'}}(y)}{\chi(\xi(x))}\indic{\hat\tau_{r'}=D_i}\bigm| B^\mathrm{C}_{\hat\tau_{r'}}\right]\notag\\
&=\sum_{i=1}^\infty\mathds{E}^{M^{r'+1}}\left[\left\{\indic{d_i=1}\left(\frac{R^\mathrm{C}_{\hat\tau_{r'-1}}(y)+P^\mathrm{C}_{\hat\tau_{r'-1}}(y)}{\chi(\xi(x))}\right)+\indic{d_i=0}\frac{R^\mathrm{C}_{\hat\tau_{r'-1}}(y)}{\chi(\xi(x))}\right\}\indic{\hat\tau_{r'}=D_i}\bigm| B^\mathrm{C}_{\hat\tau_{r'}}\right]\notag\\
&=\sum_{i=1}^\infty\mathds{E}^{M^{r'+1}}\left[\frac{R^\mathrm{C}_{\hat\tau_{r'-1}}(y)+\frac12P^\mathrm{C}_{\hat\tau_{r'-1}}(y)}{\chi(\xi(x))}\indic{\hat\tau_{r'}=D_i}\bigm| B^\mathrm{C}_{\hat\tau_{r'}}\right]\notag\\
&=\mathds{E}^{M^{r'+1}}\left[\frac{R^\mathrm{C}_{\hat\tau_{r'-1}}(y)+\frac12P^\mathrm{C}_{\hat\tau_{r'-1}}(y)}{\chi(\xi(x))}\indic{E_3}\bigm| B^\mathrm{C}_{\hat\tau_{r'}}\right]\notag\\
&=\mathds{E}^{M^{r'+1}}\left[\frac{\ink_{\hat\tau_{r'-1}}(y)}{\chi(\xi(x))}\indic{E_3}\bigm| B^\mathrm{C}_{\hat\tau_{r'}}\right]\notag\\
&=\mathds{E}^{M^{r'}}\left[\frac{\ink_{\hat\tau_{r'}}(y)}{\chi(\xi(x))}\indic{E_3}\bigm| B^\mathrm{C}_{\hat\tau_{r'}}\right].\label{e:E3}
\end{align}
Putting together equations \eqref{e:E1}--\eqref{e:E3} and using that $E_1,E_2,E_3$ form a partition, we obtain that for each $y\in V$, 
\[
\mbox{a.s.}\quad \mathds{E}^{M^{r'+1}}\left[\frac{\ink_{\hat\tau_{r'}}(y)}{\chi(\xi(x))}\bigm| B^\mathrm{C}_{\hat\tau_{r'}}\right]=\mathds{E}^{M^{r'}}\left[\frac{\ink_{\hat\tau_{r'}}(y)}{\chi(\xi(x))}\bigm| B^\mathrm{C}_{\hat\tau_{r'}}\right],
\]and thus by the inductive hypothesis, we have shown \eqref{e:mipsmr}.
\end{proof}

Next, we show how Proposition~\ref{P:MaBBtoC} can be used to bound the total variation distance between two MaBB configurations in terms of the total amount of ink in the chameleon process.

Recall from~\eqref{e:pixi} the law $\pi_\zeta$ for $\zeta\in \Omega_{G,m-1}$ and denote by $\tilde m_t$ a random variable which, conditionally on $\xi_t=\zeta$, has law $\pi_{\zeta}$. Recall also the definition of event Fill from Lemma~\ref{L:inkmart}.
\begin{proposition}\label{P:tvboundink}
Let $(\xi_t,m_t)$ denote the time-$t$ configuration of a MaBB initialised at $(\xi,x)\in\Omega'_{G,m}$. For any $t>0$,
\[
\|\mathcal{L}((\xi_t,m_t))-\mathcal{L}((\xi_t,\tilde m_t))\|_\mathrm{TV}\le 1-\,\mathds{E}\left[\frac{\ink_t^{(\xi,x)}}{a(m-1)+bn}\mid\mathrm{Fill}\right].
\]
\end{proposition}
\begin{proof}
This is similar to the proof of Lemma 8.1 of \citet{olive}. 

Recall from~\eqref{e:pixi} the law $\pi_\zeta$ for $\zeta\in \Omega_{G,m-1}$ and denote by $\tilde m_t$ a random variable which, conditionally on $\xi_t=\zeta$, has law $\pi_{\zeta}$. Recall also the definition of event Fill from Lemma~\ref{L:inkmart}.

By Proposition~\ref{P:MaBBtoC}, for any $(\zeta,y)\in\Omega'_{G,m}$,
\[
\mathds{P}(\xi_t=\zeta,m_t=y)=\mathds{E}\left[\frac{\ink_t^{(\xi,x)}(y)}{\chi(\xi(x))}\indic{B^\mathrm{C}_t=\zeta}\right]\ge\mathds{E}\left[\frac{\ink_t^{(\xi,x)}(y)}{\chi(\xi(x))}\indic{\{B^\mathrm{C}_t=\zeta\}\cap\mathrm{Fill}}\right].
\]
On the other hand, using that $B^\mathrm{C}_t$ and $\xi_t$ have the same distribution and Corollary~\ref{cor:fill},
\[
\mathds{P}(\xi_t=\zeta, \tilde m_t=y)=\pi_\zeta(y)\mathds{P}(\xi_t=\zeta)=\frac{\pi_\zeta(y)}{\mathds{P}(\mathrm{Fill})}\mathds{P}(\{B^\mathrm{C}_t=\zeta\}\cap\mathrm{Fill}).
\]
We deduce that 
\begin{equation}
\begin{aligned}\label{e:pospart}
&(\mathds{P}(\xi_t=\zeta,\,\tilde m_t=y)-\mathds{P}(\xi_t=\zeta,\,m_t=y))_+\\&\le \left(\mathds{E}\left[\indic{\{B^\mathrm{C}_t=\zeta\}\cap\mathrm{Fill}}\left(\frac{\pi_\zeta(y)}{\mathds{P}(\mathrm{Fill})}-\frac{\ink_t^{(\xi,x)}(y)}{\chi(\xi(x))}\right)\right]\right)_+.
\end{aligned}
\end{equation}
Observe that on event $\{B^\mathrm{C}_t=\zeta\}$, we have $\ink^{(\xi,x)}_t(y)\le \chi(\zeta(y))$ by Lemma~\ref{L:limits}, and so (on this event),
\[
\frac{\ink^{(\xi,x)}_t(y)}{\chi(\xi(x))}\le\frac{\chi(\zeta(y))}{\chi(\xi(x))}=\frac{\chi(\zeta(y))}{(a(m-1)+bn)\mathds{P}(\mathrm{Fill})}=\frac{\pi_\zeta(y)}{\mathds{P}(\mathrm{Fill})},
\]
where the first equality is due to Corollary~\ref{cor:fill} and the second from the definition of the colour function $\chi$. As a result we deduce from~\eqref{e:pospart} that
\begin{align*}
(\mathds{P}(\xi_t=\zeta,\,\tilde m_t=y)-\mathds{P}(\xi_t=\zeta,\,m_t=y))_+\le\mathds{E}\left[\indic{\{B^\mathrm{C}_t=\zeta\}\cap\mathrm{Fill}}\left(\frac{\pi_\zeta(y)}{\mathds{P}(\mathrm{Fill})}-\frac{\ink^{(\xi,x)}_t(y)}{\chi(\xi(x))}\right)\right].
\end{align*}

We take a sum over $y$ followed by $\zeta$ to obtain
\begin{align*}
\|\mathcal{L}((\xi_t,m_t))-\mathcal{L}((\xi_t,\tilde m_t))\|_\mathrm{TV}&\le\mathds{E}\left[\indic{\mathrm{Fill}}\left(\frac{1}{\mathds{P}(\mathrm{Fill})}-\frac{\ink^{(\xi,x)}_t}{\chi(\xi(x))}\right)\right]\\
&=1-\mathds{P}(\mathrm{Fill})\,\mathds{E}\left[\frac{\ink_t^{(\xi,x)}}{\chi(\xi(x))}\mid\mathrm{Fill}\right]\\
&=1-\,\mathds{E}\left[\frac{\ink_t^{(\xi,x)}}{a(m-1)+bn}\mid\mathrm{Fill}\right],
\end{align*}
using Lemma~\ref{L:inkmart} in the last step.
\end{proof}

Recall from Section~\ref{SS:evol} that for each $\ell\in\mathds{N}$,  $\widehat{\ink}_\ell$ denotes the value of ink just after the $\ell$th depinking time. We write $\widehat{\ink}_\ell^{(\xi,x)}$ to emphasise the dependence on the initial configuration of the corresponding MaBB.
\begin{lemma}\label{L:expdec}
Fix $(\xi,x)\in\Omega'_{G,m}$. For each $\ell\ge1$,
\[
1-\mathds{E}\left[\frac{\widehat{\ink}_\ell^{(\xi,x)}}{a(m-1)+bn}\mid\mathrm{Fill}\right]\le (71/72)^\ell\sqrt{a(m-1)+bn}.
\]
\end{lemma}
We omit the proof (which uses Lemma~\ref{L:inkmc}) of this result since it is identical to the proof of Proposition 6.1 in \cite{olive}, except that here ink can take values in $\{0,\ldots,a(m-1)+bn\}$ (in contrast with \cite{olive} in which $\ink\in\{0,\ldots,n\}$).

\section{Expected loss of red in a round}\label{s:loss}
In this section we show that during a single round (which starts with fewer red particles than white) the number of red particles decreases in expectation by a constant factor. 

Let $M_{i,j}(G)$ denote the meeting time of two independent random walks started from vertices $i$ and $j$ on $G$ and recall that $\hat M_{i,j}(G)$ denotes the meeting time of two independent random walks started from vertices $i$ and $j$ on the graph obtained from $G$ by halving the edge-weights, that is, $\hat M_{i,j}(G)=M_{i,j}(\hat G)$.

Consider a slight modification to the chameleon process in which we replace the number of selected particles~\eqref{e:cond} in Step 2 with $k$, that is, we allow all paired reds particles to be pinkened. We call this the \emph{modified} chameleon process.

\begin{proposition}\label{P:lossred}
Suppose the modified chameleon process  starts a round with red configuration $R$, white configuration $W$ and black configuration $B$ such that $|R|\le |W|$. If the round length $T$ satisfies $T\ge 2\max_{i,j}\mathds{E}\hat M_{i,j}(G)$ then $\mathds{E}[|R_{T-}^\mathrm{C}|]\le (1-c)|R|$, with $c=\frac{(p^*)^2}{4a}$.
\end{proposition}
\begin{remark}If instead  $|W|\le |R|$ then we have an equivalent result:  $\mathds{E}[|W_{T-}^\mathrm{C}|]\le (1-c)|W|$.
\end{remark}
\begin{proof}
We shall only count pinkenings between paired red and white particles which get coloured pink the first time they meet (if they do) during the round. (This means that we do not have to worry about how the particles move after their first meeting time -- they no longer move independently once they meet.)

Since we assume $|R|\le |W|$, all red particles will have a label in $\{1,\ldots,|R|\}$. Let $M^r$ denote the meeting time of red particle with label $r$ with its paired white particle; this is the first time the two particles are on the same ringing edge. If two paired particles start the round on the same vertex we set their meeting time to be the first time this vertex is on a ringing edge. For each $s\in\mathds{N}$ write $F_s(r)$ for the event that a red particle with label $r$ remains in the pooled pile after Step 1 of the update at time $\tau_s$ (if red particle with label $r$ is not on edge $e_r$ at time $\tau_r-$, we set $F_r(s)=\varnothing$), and write $G_s$ for the event that we do Step 3a (rather than Step 3b) at the update at time $\tau_s$. 

We also write $e_s^1, \,e_s^2$ for the two vertices on edge $e_s$ (in an arbitrary order), $u_s(e_s^1)$ and $\ell_s(e_s^1)$ for the values of $u(e_s^1)$ and $\ell(e_s^2)$ at the update at time $\tau_s$, and $\theta_s(e_s^1)$ for the probability $\theta(e_s^1)$ at the update time $\tau_s$. 

We lower-bound the expected number of pink particles created during a single round (which has length $T$) of the modified chameleon process in which at the start of the round the configuration of red particles is $R$ by
\begin{align*}
2\mathds{E}\left[\sum_{r=1}^{|R|}\sum_{s=1}^\infty\indic{M^{r}=\tau_s<T}\indic{F_s(r)}\indic{G_s}\right]&=2\sum_{r=1}^{|R|}\sum_{s=1}^\infty\mathds{E}\left[\indic{M^r=\tau_s<T}\mathds{P}(F_s(r)\cap G_s\mid \tau_s, M^r)\right].
\end{align*}

Observe that conditionally on the configuration of the chameleon process at time $\tau_s-$ and the configuration of black particles at time $\tau_r$,  $F_s(r)$ and $G_s$ are independent since $F_s(r)$ depends further only on the randomness at Step 1, and $G_s$ the randomness at Step 2 (and we have constructed the chameleon process so that these are independent). Therefore we have almost surely
\begin{align}\label{e:FintG}
\mathds{P}(F_s(r)\cap G_s\mid \tau_s, M^r)=\mathds{E}[2 (\theta_s(e_s^1)\wedge (1-\theta_s(e_s^1))\indic{F_s(r)}\mid \tau_s, M^r].
\end{align}

Next, for each $s\in\mathds{N}$, we introduce an event $A_s$  which:
\begin{enumerate}
\item has probability $p^*$ (recall this constant comes from Property~\ref{a:prob}),
\item prescribes only the value that $U_s^b$ takes,
\item on event $A_s$, for each $i\in\{1,2\}$, given $e_s$ and $B^\mathrm{C}_{\tau_s-}$, configuration $B^\mathrm{C}_{\tau_s}$ satisfies almost surely \begin{enumerate}
\item  $\mathrm{P}_{e_s,B^\mathrm{C}_{\tau_s-},B^\mathrm{C}_{\tau_s}}(e_s^i,e_s^i)\in[p^*\wedge\frac29,(1-p^*)\vee\frac79]$,
\item  $\chi(B^\mathrm{C}_{\tau_s}(e_s^1))/\chi(B^\mathrm{C}_{\tau_s}(e_s^2))\in [1/(2a),2a]$.
\end{enumerate}
\end{enumerate} We suppose for now that such an event exists.
As $A_s$ only prescribes $U_s^b$, it is independent of events $\{M^r=\tau_s\}$ and $\{\tau_s<T\}$ (which do not depend on $U_s^b$). Thus  from~\eqref{e:FintG} and by Lemma~\ref{L:thetabounds} we have
\[
\mathds{P}(F_s(r)\cap G_s\mid \tau_s, M^r)\ge 2(\frac29\wedge p^*) p^*\, \mathds{P}(F_s(r)\mid \tau_s, M^r, A_s)\ge\frac43(p^*)^2\,\mathds{P}(F_s(r)\mid \tau_s, M^r, A_s),
\]using $p^*<1/3$.
We also have that  $\mathds{P}(F_s(r)\mid \tau_s, M^r, A_s)\ge 1/(2a)$ almost surely. This follows from Lemma~\ref{L:pairres} by first conditioning on the configuration of the chameleon process at time $\tau_s-$, since given this, $\indic{F_s(r)}$ is independent of $M^r$ and $\tau_s$. Thus
our lower-bound on the expected number of pink particles created becomes
\[
\frac{4}{3a}(p^*)^2 \sum_{r=1}^{|R|}\sum_{s=1}^\infty\mathds{P}\left(M^r=\tau_s<T\right)\ge \frac{(p^*)^2}{a}|R|\min_r\mathds{P}(M^r<T).
\]
For a red and white pair (with label $r$) on the same vertex $v$, say, at the start of the round, $M^r$ is the first time $v$ is on a ringing edge. Suppose $w$ is a neighbour of $v$ (chosen arbitrarily) and recall $M_{v,w}(\hat G)$ is the meeting time of two random walks  on $\hat G$ started from vertices $v$ and $w$ respectively. Then $\mathds{P}(M^r< T)\ge\frac12\mathds{P}(M_{v,w}< T)$, since for the random walks to meet, vertex $v$ must be on a ringing edge for at least one of the two random walk processes. Then by Markov's inequality, we have in this case that $\mathds{P}(M^r< T)\ge\frac12(1-\max_{i,j}\mathds{E}M_{i,j}(\hat G)/T)$. On the other hand, for a red and white pair which start the round on different vertices, we can directly apply Markov's inequality  to obtain $\mathds{P}(M^r< T)\ge 1-\max_{i,j}\mathds{E}M_{i,j}(\hat G)/T$.

Thus if $T\ge2 \max_{i,j}\mathds{E}\hat M_{i,j}(G)$ then  for any $r$, $\mathds{P}(M^r< T)\ge 1/4$, which shows that statement of the proposition with $c=\frac{(p^*)^2}{4a}$.

It remains to demonstrate the existence of the event $A_s$ for each $s\in\mathds{N}$.
We set $A_s=\{U_s^b\le p^*\},$ which clearly has probability $p^*$ and only prescribes the value that $U_s^b$ takes. 
Recall from the discussion in Section~\ref{S:graphical} (in particular~\eqref{e:ulesspimplication}) that $U_s^b\le p^*$ implies that if there are at least two non-marked particles  on $e_s$ then a proportion in $[1/3,2/3]$ of the non-marked particles on the edge end up on each vertex on $e_s$ (at time $\tau_s$).

Thus on event $A_s$ (and as black particles in the chameleon process move as non-marked particles in MaBB), almost surely,
\begin{align*}
B_{\tau_s}^\mathrm{C}(e_s^1)&\ge \frac13\indic{\sum_{i=1}^2B_{\tau_s-}^\mathrm{C}(e_s^i)\ge 2}\sum_{i=1}^2B_{\tau_s-}^\mathrm{C}(e_s^i),\\
B_{\tau_s}^\mathrm{C}(e_s^2)&\le \frac23\indic{\sum_{i=1}^2B^\mathrm{C}_{\tau_s-}(e_s^i)\ge 2}\sum_{i=1}^2B_{\tau_s-}^\mathrm{C}(e_s^i)+\indic{\sum_{i=1}^2B_{\tau_s-}^\mathrm{C}(e_s^i)=1}\\
&\le \frac23\indic{\sum_{i=1}^2B^\mathrm{C}_{\tau_s-}(e_s^i)\ge 2}\sum_{i=1}^2B^\mathrm{C}_{\tau_s-}(e_s^i)+1.
\end{align*}
Thus on event $A_s$,  almost surely,
\begin{align*}
\frac{\chi(B^\mathrm{C}_{\tau_s}(e_s^1))}{\chi(B^\mathrm{C}_{\tau_s}(e_s^2))}&\ge \gamma+\frac{b(1-\gamma)-a\gamma}{\frac{2a}{3}\indic{\sum_{i=1}^2B_{\tau_s-}^\mathrm{C}(e_s^i)\ge 2}\sum_{i=1}^2B_{\tau_s-}^\mathrm{C}(e_s^i)+a+b}\\
&\ge \gamma,
\end{align*}
for any $\gamma\le 1/2$ provided $b(1-\gamma)\ge a\gamma$. We similarly have $\frac{\chi(B^\mathrm{C}_{\tau_s}(e_s^2))}{\chi(B^\mathrm{C}_{\tau_s}(e_s^1))}\ge \gamma$ under the same condition. This condition is satisfied taking $\gamma=1/(2a)$ (and this is indeed $\le 1/2$ as $a\ge 1$).

Finally, it remains to show that for each $i\in\{1,2\}$ we have $\mathrm{P}_{e_s,B^\mathrm{C}_{\tau_s-},B^\mathrm{C}_{\tau_s}}(e_s^i,e_s^i)\in[p^*\wedge\frac29,(1-p^*)\vee\frac79]$ on event $A_s$, almost surely. This is the probability that in the MaBB process, if the marked particle is on vertex $e_s^i$, it remains on vertex $e_s^i$ given the non-marked particles update from configuration $B^\mathrm{C}_{\tau_s-}$ to $B^\mathrm{C}_{\tau_s}$ when edge $e_s$ rings. 

Suppose $\sum_{j=1}^2 B^\mathrm{C}_{\tau_s-}(e_s^j)\ge 2$, i.e.\! before the update there are at least 2 black particles on $e_s$. For $y\in e_s$, write $m_s(y)\in\{0,1\}$ for the number of marked particles on $y$ after the update at time $\tau_s$. On event $A_s$, for each $i\in\{1,2\}$ we have $B^\mathrm{C}_{\tau_s}(e_s^i)\in[\frac13\sum_{j=1}^2 B^\mathrm{C}_{\tau_s-}(e_s^j),\frac23\sum_{j=1}^2 B^\mathrm{C}_{\tau_s-}(e_s^j)]$ and thus
\begin{align*}
B^\mathrm{C}_{\tau_s}(e_s^i)+m_s(e_s^i)&\in\Big[\frac13\sum_{j=1}^2 B^\mathrm{C}_{\tau_s-}(e_s^j),\frac23\sum_{j=1}^2 B^\mathrm{C}_{\tau_s-}(e_s^j)+1\Big]\\&=\Big[\frac13\big(\sum_{j=1}^2 B^\mathrm{C}_{\tau_s-}(e_s^j)+1\big)-\frac13,\frac23\big(\sum_{j=1}^2 B^\mathrm{C}_{\tau_s-}(e_s^j)+1\big)+\frac13\Big]\\
&\subseteq\Big[\frac29\big(\sum_{j=1}^2 B^\mathrm{C}_{\tau_s-}(e_s^j)+1\big),\frac79\big(\sum_{j=1}^2 B^\mathrm{C}_{\tau_s-}(e_s^j)+1\big)\Big].
\end{align*}
Now recall (from the discussion after~\eqref{e:MaBB2BBSP}) the description of the MaBB process in which we remove the mark on the marked particle, then update as the BBSP, and then choose a uniform particle on the edge on which to apply the mark. Together with the just-determined bound on the number of particles on $e_s^i$, this tells us that the probability the marked particle is on $e_s^i$ after the update is in $[\frac29,\frac79]$.

Suppose now that $\sum_{j=1}^2 B^\mathrm{C}_{\tau_s-}(e_s^j)=1$. Recall the definition of $\mathrm{P}_{e,B^\mathrm{C}_{\tau_s-},B^\mathrm{C}_{\tau_s}}(e_s^i,e_s^i)$ as
\begin{align*}
\mathrm{P}_{e,B^\mathrm{C}_{\tau_s-},B^\mathrm{C}_{\tau_s}}(e_s^i,e_s^i)&:=\frac{B^\mathrm{C}_{\tau_s}(e_s^i)+1}{B^\mathrm{C}_{\tau_s}(e_s^1)+B^\mathrm{C}_{\tau_s}(e_s^2)+1}\frac{\mathrm{P}_e^{\mathrm{BB}(G,s,m)}(C_{B^\mathrm{C}_{\tau_s-},e_s^i},C_{B^\mathrm{C}_{\tau_s},e_s^i})}{\mathrm{P}_e^{\mathrm{BB}(G,s,m-1)}(B^\mathrm{C}_{\tau_s-},B^\mathrm{C}_{\tau_s})}\\
&=\frac{B^\mathrm{C}_{\tau_s}(e_s^i)+1}{2}\frac{\mathrm{P}_e^{\mathrm{BB}(G,s,m)}(C_{B^\mathrm{C}_{\tau_s-},e_s^i},C_{B^\mathrm{C}_{\tau_s},e_s^i})}{\frac12}\\
&=(B^\mathrm{C}_{\tau_s}(e_s^i)+1)\mathrm{P}_e^{\mathrm{BB}(G,s,m)}(C_{B^\mathrm{C}_{\tau_s-},e_s^i},C_{B^\mathrm{C}_{\tau_s},e_s^i}),
\end{align*}
where we have used Property~\ref{a:sym} to obtain $\mathrm{P}_e^{\mathrm{BB}(G,s,m-1)}(B^\mathrm{C}_{\tau_s-},B^\mathrm{C}_{\tau_s})=1/2$. BBSP configuration $C_{B^\mathrm{C}_{\tau_s-},e_s^i}$ has two particles, thus by Property~\ref{a:sym} and the second part of Property~\ref{a:prob}, $\mathrm{P}_e^{\mathrm{BB}(G,s,m)}(C_{B^\mathrm{C}_{\tau_s-},e_s^i},C_{B^\mathrm{C}_{\tau_s},e_s^i})\ge p^*$, and so $\mathrm{P}_{e,B^\mathrm{C}_{\tau_s-},B^\mathrm{C}_{\tau_s}}(e_s^i,e_s^i)\ge p^*$. If $B^\mathrm{C}_{\tau_s}(e_s^i)=0$ (so that the non-marked particle and the marked particle end up on different vertices) we have (again by Property~\ref{a:prob}) $\mathrm{P}_{e,B^\mathrm{C}_{\tau_s-},B^\mathrm{C}_{\tau_s}}(e_s^i,e_s^i)\le 1-2p^*$, whereas if $B^\mathrm{C}_{\tau_s}(e_s^i)=1$, then by Properties~\ref{a:sym} and~\ref{a:prob}, $\mathrm{P}_{e,B^\mathrm{C}_{\tau_s-},B^\mathrm{C}_{\tau_s}}(e_s^i,e_s^i)\le (B^\mathrm{C}_{\tau_s}(e_s^i)+1)\frac{1-p^*}{2}=1-p^*$.

Finally, if $B^\mathrm{C}_{\tau_s-}(e_s^i)=0$, then a marked particle on $e_s^i$ stays on $e_s^i$ at the update time $\tau_s$ with probability $1/2$ by Property~\ref{a:sym}.

Thus in all cases we have that on event $A_s$, almost surely $\mathrm{P}_{e_s,B^\mathrm{C}_{\tau_s-},B^\mathrm{C}_{\tau_s}}(e_s^i,e_s^i)\in[p^*\wedge\frac29, (1-p^*)\vee\frac79]$.\qedhere
\end{proof}
\section{Proof of Proposition~\ref{P:cham}}\label{S:proof}
We can now put together the results obtained so far and complete the proof of Proposition~\ref{P:cham}. These arguments are similar to those in previous works using a chameleon process.

The next result bounds the first depinking time $D_1$. We wish to apply this result for any of the depinking times, and so we present the result in terms of a chameleon process started from any configuration in $\mathcal{C}(m)$. In reality, a chameleon process at time 0 will always have all red particles on a single vertex, as is apparent from Definition~\ref{d:cham}.

\begin{lemma}\label{L:depinkexp}
If the round length $T$ satisfies $T\ge 2\max_{i,j}\mathds{E}\hat M_{i,j}(G)$, then from any initial configuration in $\mathcal{C}(m)$ of the (non-modified) chameleon process, the first depinking time has an exponential moment:
\[
\mathds{E}[e^{D_1/(KT)}]\le \frac{12a}{(p^*)^2},
\]
where $K=8a/(p^*)^2$.
\end{lemma}
\begin{proof}
This proof follows closely the proof of \cite[Lemma 9.2]{olive}. By the same arguments as there, we obtain 
\[
\mathds{P}(D_1>iT)\le \frac32 (1-c)^i,
\]
for any integer $i\ge 1$, with $c=(p^*)^2/(4a)$ the constant from Proposition~\ref{P:lossred}.

To obtain the bound on the exponential moment, observe that for any $K>0$, 
\begin{align*}
\mathds{E}[e^{D_1/(KT)}]&=\sum_{i=1}^\infty\mathds{E}[e^{D_1/(KT)}\indic{iT<D_1<(i+1)T}]\\
&\le \sum_{i=1}^\infty\mathds{E}[e^{(i+1)/K}\indic{D_1>iT}]=\sum_{i=1}^\infty e^{(i+1)/K}\mathds{P}(D_1>iT)\\
&\le \sum_{i=1}^\infty \frac32e^{1/K}\exp\left(\frac{i}{K}+i\log(1-c)\right).
\end{align*}
Set $K=2/c\ge -2/\log(1-c)$; then $i/K+i\log(1-c)\le \frac{i}{2}\log(1-c)<0$, and 
\[
\mathds{E}[e^{D_1/(KT)}]\le \frac{3}{2(1-\sqrt{1-c})}\le \frac{3}{c}.\qedhere
\]
\end{proof}

We now show a result which bounds the exponential moment of the $j$th depinking time. In order to emphasise the initial configuration on the underlying MaBB we shall write $D_j((\xi,x))$ for the $j$th depinking time of a chameleon process corresponding to a MaBB which at time 0 is in configuration $(\xi,x)\in\Omega'_{G,m}$.
\begin{lemma}\label{L:expmoment}
If the round length $T$ satisfies $T\ge 2\max_{i,j}\mathds{E}\hat M_{i,j}(G)$, then for any $(\xi,x)\in\Omega'_{G,m}$,  for all $j\in\mathds{N}$, \[\mathds{E}[e^{D_j((\xi,x))/(K T)}\mid \mathrm{Fill}]\le \left(\frac{12a}{(p^*)^2}\right)^j,\]
where $K=8a/(p^*)^2$.
\end{lemma}
\begin{proof}
This proof follows identically to the proof of Lemma~6.2 from~\cite{olive} and uses Lemma~\ref{L:depinkexp}.
\end{proof}

\begin{proof}[Proof of Proposition~\ref{P:cham}]

We apply Proposition~\ref{P:tvboundink} to obtain
\begin{align}\label{e:4}
\|\mathcal{L}((\xi_t,m_t))-\mathcal{L}((\xi_t,\tilde m_t))\|_\mathrm{TV}\le \max_{(\xi,x)\in\Omega_{G,m}'}\mathds{E}\left[1-\frac{\ink_t^{(\xi,x)}}{a(m-1)+bn}\mid\mathrm{Fill}\right].
\end{align}
Lemma~\ref{L:inkonly} says that the total ink can only change at depinking times, thus (recalling the definition of $\widehat{\ink}_j$), $\ink_t^{(\xi,x)}=\chi(\xi(x))$ if $t<D_1((\xi,x))$ and $\ink_t^{(\xi,x)}=\widehat{\ink}_j^{(\xi,x)}$ if $D_j((\xi,x))\le t<D_{j+1}((\xi,x))$ for some $j$. Hence we have that for any $j\ge 1$,
\begin{align*}
1-\frac{\ink_t^{(\xi,x)}}{a(m-1)+bn}&\le \max_{\ell\ge j}\left(1-\frac{\widehat{\ink}_\ell^{(\xi,x)}}{a(m-1)+bn}\right)+\indic{t<D_j((\xi,x))}\\
&\le \sum_{\ell\ge j}\left(1-\frac{\widehat{\ink}_\ell^{(\xi,x)}}{a(m-1)+bn}\right)+\indic{t<D_j((\xi,x))}.
\end{align*}Taking expectations (given Fill) on both sides and using~\eqref{e:4} we obtain for any $j\ge1$,
\begin{align*}
&\|\mathcal{L}((\xi_t,m_t))-\mathcal{L}((\xi_t,\tilde m_t))\|_\mathrm{TV}\\&\le \max_{(\xi,x)\in\Omega_{G,m}'}\left\{\sum_{\ell\ge j}\mathds{E}\left[1-\frac{\widehat{\ink}_\ell^{(\xi,x)}}{a(m-1)+bn}\mid\mathrm{Fill}\right]+\mathds{P}(D_j((\xi,x))>t\mid\mathrm{Fill})\right\}.
\end{align*}
We bound the first term using Lemma~\ref{L:expdec} to obtain
\begin{align*}
&\|\mathcal{L}((\xi_t,m_t))-\mathcal{L}((\xi_t,\tilde m_t))\|_\mathrm{TV}\\&\le  \max_{(\xi,x)\in\Omega_{G,m}'}\left\{\sum_{\ell\ge j}(71/72)^\ell\sqrt{a(m-1)+bn}+\mathds{P}(D_j((\xi,x))>t\mid\mathrm{Fill})\right\}\\
&\le 100e^{-j\log(72/71)}\sqrt{a(m-1)+bn}+2m\max_{(\xi,x)\in\Omega_{G,m}'}\mathds{P}(D_j((\xi,x))>t\mid\mathrm{Fill}),
\end{align*}and then by a Chernoff bound and Lemma~\ref{L:expmoment} (recall constant $K=8a/(p^*)^2$) we have that 
\begin{align*}
&\|\mathcal{L}((\xi_t,m_t))-\mathcal{L}((\xi_t,\tilde m_t))\|_\mathrm{TV}\\&\le100e^{-j\log(72/71)}\sqrt{a(m-1)+bn}+2me^{j\log(12a(p^*)^{-2})-t/(2K \max_{i,j}\mathds{E}\hat M_{i,j}(G))}.
\end{align*}
This holds for all $j\ge1$ so if we apply it with $j=\left\lfloor\frac{t}{4K\max_{i,j}\mathds{E}\hat M_{i,j}(G)\log(12a(p^*)^{-2})}\right\rfloor$ we obtain
\begin{align*}
\|\mathcal{L}((\xi_t,m_t))-\mathcal{L}((\xi_t,\tilde m_t))\|_\mathrm{TV}&\le K_0  e^{-t/(4K\max_{i,j}\mathds{E}\hat M_{i,j}(G)\log(12a(p^*)^{-2}))}\sqrt{a(m-1)+bn}
\end{align*}
for some universal constant $K_0>0$. 
\end{proof}
\section{Lower bound for the line}\label{S:lower}
In this section we prove Proposition~\ref{P:linelower}, our mixing time lower bound for the line. Our argument is essentially a discretization of an argument from~\citet{10.1214/20-AOP1428} for a continuous-mass redistribution model on the line, but we include the details here for completeness. 

The first step establishes a coupling between two realisations of the beta-binomial splitting process on the line in which we remove the restriction on the number of particles so that the state space of the processes is $\mathds{N}_0^{L_n}$. We write such a process as BB$(L_n,s)$ and by detailed balance considerations it is immediate that reversible measures for this process are those for which $\xi(k)$ are i.i.d.\! Negative Binomial random variables with parameters $s$ and any $p\in[0,1]$. The coupling has a certain monotonicity property as demonstrated by the following result.
\begin{lemma}\label{L:mon}
Let $(\xi_t)_{t\ge0}$ and $(\xi'_t)_{t\ge0}$ be realisations of BB$(L_n,s)$ such that $\xi_0(k)\le \xi'_0(k)$ for all $k\in L_n$.
There exists a coupling of these realisations such that for all $t\ge0$ and $k\in L_n$, $\xi_t(k)\le \xi'_t(k)$.
\end{lemma}
\begin{proof}
We use the same edge ringing times in the two processes. Suppose edge $e=\{k,k+1\}$ rings at time $s$ and $\xi'_{s-}(k)+\xi'_{s-}(k+1)\ge \xi_{s-}(k)+\xi_{s-}(k+1)$. It is then immediate that we can couple the two processes so that $\xi'_s(k)\ge \xi_s(k)$ and $\xi'_s(k+1)\ge \xi_s(k+1)$ (for instance we can generate a  Beta Binomial$(r,s,s)$ random variable by first sampling $p\sim\,$Beta$(s,s)$ and then sampling a Bin$(r,p)$; we can couple the updates by choosing the same $p$ for the two processes). The proof is then complete by induction on the update times.
\end{proof}

The next step is the identification of an eigenfunction of the generator $\mathcal{L}^{\mathrm{BB}(L_n,s,m)}$.
\begin{lemma}\label{L:eigen}
The map
\[
f(\xi):=\sum_{k=1}^{n-1}\sin\left(\frac{\pi k}{n}\right)\left(\sum_{i=1}^k\xi(i)-\frac{mk}{n}\right)
\]is an eigenfunction of $\mathcal{L}^{\mathrm{BB}(L_n,s,m)}$ with eigenvalue $\lambda:=\frac1{n-1}(\cos(\pi/n)-1)$.
\end{lemma}
\begin{proof}
Since $\sum_{i=1}^k\xi(i)$ only updates when edge $\{k,k+1\}$ rings, the action of the generator on the map $x_k(\xi):=\sum_{i=1}^k\xi(i)$, $k=1,\ldots,n-1$, $\xi\in\Omega_{L_n,m}$ (and set $x_0\equiv 0$), is given by
\begin{align*}
(\mathcal{L}^{\mathrm{BB}(L_n,s,m)} x_k)(\xi)&=\frac{n-2}{n-1}x_k(\xi)+\frac1{n-1}\left(\frac12 x_{k-1}(\xi)+\frac12 x_{k+1}(\xi)\right)-x_k(\xi)\\
&=\frac1{2(n-1)}\left(x_{k-1}(\xi)+x_{k+1}(\xi)-2x_k(\xi)\right).
\end{align*}
It then follows by summation by parts that 
\[
f(\xi)=\sum_{k=1}^{n-1}\sin\left(\frac{\pi k}{n}\right)\left(x_k-\frac{mk}{n}\right)
\]is an eigenfunction of $\mathcal{L}^{\mathrm{BB}(L_n,s,m)}$ with eigenvalue $\frac1{n-1}(\cos(\pi/n)-1)$.
\end{proof}
The argument proceeds by using Wilson's method~\citep{wilson2004mixing} applied to the eigenfunction $f$ from Lemma~\ref{L:eigen}. In particular, with $(\xi_t)_{t\ge0}$ a realisation of BB$(L_n,s,m)$, we show that $f(\xi_t)$ is far from the equilibrium value $\sum_\xi f(\xi)\pi^{\mathrm{BB}(L_n,s,m)}(\xi)$ with high probability at time $t=\frac{n^3}{\pi^2}\left(\log n-\log \left(1+\frac{n}{m}+\frac1{s}\right)-C_\varepsilon\right)$. The initial value $\xi_0$ is chosen explicitly in terms of Negative Binomial random variables; precisely, we let $\xi_0(1),\ldots,\xi_0(\lfloor n/2\rfloor)$ be i.i.d.\! NegBin$(s,\frac{\lfloor n/2\rfloor s}{m+\lfloor n/2\rfloor s})$ random variables conditioned on $\xi_0(1)+\cdots+\xi_0(\lfloor n/2\rfloor)=m$, and $\xi_0(k)=0$ for $k\in\{\lfloor n/2\rfloor+1,\ldots,n\}$.

\begin{lemma}\label{L:fbounds}
For the realisation $(\xi_t)_{t\ge0}$ of BB$(L_n,s,m)$ described above,  there exists a constant $c>0$ such that for all $t\ge0$ and $n$ sufficiently large, \[
\mathds{E}[f(\xi_t)]\ge c^{-1}nme^{\lambda t},\quad\mathrm{Var}(f(\xi_t))\le cnm^2\left(1+\frac{n}{m}+\frac1{s}\right).
\]
\end{lemma}
\begin{proof}We make appropriate modifications to the proof of \citet[Lemma 12]{10.1214/20-AOP1428}.
Fix time $t\ge0.$ As $f$ is an eigenfunction of $\mathcal{L}^{\mathrm{BB}(L_n,s,m)}$ we have \[
\mathds{E}[f(\xi_t)]=\mathds{E}[f(\xi_0)]e^{\lambda t}=e^{\lambda t}\sum_{k=1}^{n-1}\sin\left(\frac{\pi k}{n}\right)\left(\min\left\{\frac{mk}{\lfloor n/2\rfloor},m\right\}-\frac{mk}{n}\right)\ge e^{\lambda t}nm/20
\]for $n$ sufficiently large.

For the variance bound, observe that for $u\in[0,t]$, process $M_u:=e^{\lambda (t-u)}f(\xi_u)$ is a martingale. Thus Var$(f(\xi_t))=\mathds{E}[\langle M\rangle_t]$, where $\langle M\rangle_u$, $u\in[0,t]$, is the angle bracket process. If edge $\{k,k+1\}$ rings at time $u$, $f(\xi_u)$ changes by at most $\xi_u(k)+\xi_u(k+1)$, and thus 
\[
\partial_u\mathds{E}[\langle M\rangle_u]\le e^{2\lambda(t-u)}\sum_{k=1}^{n-1}\frac1{n-1}\mathds{E}[(\xi_u(k)+\xi_u(k+1))^2]\le \frac{4e^{2\lambda(t-u)}}{n-1}\sum_{k=1}^{n}\mathds{E}[(\xi_u(k))^2].
\]
It follows that
\[
\mathrm{Var}(f(\xi_t))\le \frac{8}{n}\int_0^t e^{2\lambda(t-u)}\sum_{k=1}^{n}\mathds{E}[(\xi_u(k))^2]\,\mathrm{d}u.
\]
We can now bound $\sum_{k=1}^{n}(\xi_u(k))^2$ using the monotone coupling of Lemma~\ref{L:mon}. Consider a realisation $(\xi'_t)_{t\ge0}$ of BB$(L_n,s)$ with initial values $\{\xi'_0(k)\}_{k\in L_n}$ being i.i.d.\! NegBin$(s,\frac{\lfloor n/2\rfloor s}{m+\lfloor n/2\rfloor s})$ random variables conditioned on $\xi'_0(1)+\cdots+\xi'_0(\lfloor n/2\rfloor)\ge m$. Further, these initial values can be coupled to $\{\xi_0(k)\}_{k\in L_n}$ so that $\xi_0(k)\le \xi'_0(k)$ for all $k\in L_n$. Lemma~\ref{L:mon} thus gives that $\xi_u(k)\le \xi'_u(k)$ for all $k\in L_n$ and $u\ge0$. In particular we have 
$\mathds{E}[(\xi_u(k))^2]\le \mathds{E}[(\xi'_u(k))^2]$ for all $k\in L_n$ and $u\in[0,t]$.

Finally we consider a realisation $(\xi''_t)_{t\ge0}$ of BB$(L_n,s)$ with initial values $\{\xi''_0(k)\}_{k\in L_n}$ being i.i.d.\! NegBin$(s,\frac{\lfloor n/2\rfloor s}{m+\lfloor n/2\rfloor s})$ random variables (without any conditioning). Then for each $k\in L_n$ and $u\in[0,t]$,
\[
\mathds{E}[(\xi'_u(k))^2]=\mathds{E}\left[(\xi''_u(k))^2\biggm| \sum_{i=1}^{\lfloor n/2\rfloor}\xi''_0(i)\ge m\right]\le \frac{\mathds{E}\left[(\xi''_u(k))^2\right]}{\mathds{P}(\sum_{i=1}^{\lfloor n/2\rfloor}\xi''_0(i)\ge m)}.
\]
As $\xi''_u$ is stationary for BB$(L_n,s)$, we have
\[
\mathds{E}[(\xi''_u(k))^2]=\left(\frac{m}{\lfloor n/2\rfloor}\right)^2\left(1+\frac{\lfloor n/2\rfloor}{m}+\frac1{s}\right).
\]Furthermore, $\mathds{E}[\xi''_u(k)]=m/\lfloor n/2\rfloor$, and so $\mathds{P}(\sum_{i=1}^{\lfloor n/2\rfloor}\xi''_0(i)\ge m)\ge 1/3$ by the central limit theorem, for $n$ sufficiently large. Putting things together, we deduce that for $n$ sufficiently large,
\begin{align*}
\mathrm{Var}(f(\xi_t))&\le 24\int_0^t e^{2\lambda(t-u)}\left(\frac{m}{\lfloor n/2\rfloor}\right)^2\left(1+\frac{\lfloor n/2\rfloor}{m}+\frac1{s}\right)\,\mathrm{d}u\\
&\le 24\left(-\frac1{2\lambda}\right)\left(\frac{m}{\lfloor n/2\rfloor}\right)^2\left(1+\frac{\lfloor n/2\rfloor}{m}+\frac1{s}\right)\\
&\le 100nm^2\left(1+\frac{n}{m}+\frac1{s}\right),
\end{align*}
uniformly in $t\ge0$.
\end{proof}
\begin{proof}[Proof of Proposition~\ref{P:linelower}]
We follow the proof of \citet[Proposition 11]{10.1214/20-AOP1428}. For each $t\ge0$, define a set
\[
E_t:=\{\xi\in \Omega_{L_n,m}:\,f(\xi)\ge \frac12\mathds{E}[f(\xi_t)]\},
\]
where $\xi_0$ is chosen as described in the paragraph before Lemma~\ref{L:fbounds}. For all $t\ge0$, we have
\[
\|\mathcal{L}(\xi_t)-\pi^{\mathrm{BB}(L_n,s,m)}\|_{\mathrm{TV}}\ge \mathds{P}(\xi_t\in E_t)-\pi^{\mathrm{BB}(L_n,s,m)}(E_t).
\]By Lemma~\ref{L:fbounds} and Chebyshev's inequality,
\[
\mathds{P}(\xi_t\in E_t)\ge 1-\frac{4\mathrm{Var}(f(\xi_t))}{(\mathds{E}[f(\xi_t)])^2}\ge1-4c^3n^{-1}e^{-2\lambda t}\left(1+\frac{n}{m}+\frac1{s}\right). 
\]
On the other hand, as $\pi^{\mathrm{BB}(L_n,s,m)}(f)=0$, we have \[\pi^{\mathrm{BB}(L_n,s,m)}(f^2)=\mathrm{Var}_{\pi^{\mathrm{BB}(L_n,s,m)}}(f)=\lim_{t\to\infty}\mathrm{Var}(f(\xi_t))\le cnm^2\left(1+\frac{n}{m}+\frac1{s}\right),
\]so by Markov's inequality,
\[
\pi^{\mathrm{BB}(L_n,s,m)}(E_t)\le \frac{4\pi^{\mathrm{BB}(L_n,s,m)}(f^2)}{(\mathds{E}[f(\xi_t)])^2}\le 4c^3e^{-2\lambda t}\left(1+\frac{n}{m}+\frac1{s}\right).
\]Thus
\[
\|\mathcal{L}(\xi_t)-\pi^{\mathrm{BB}(L_n,s,m)}\|_{\mathrm{TV}}\ge 
1-8c^3n^{-1}e^{-2\lambda t}\left(1+\frac{n}{m}+\frac1{s}\right).\]
It follows that a lower bound on the $\varepsilon$-total variation mixing time is $\frac{n^3}{\pi^2}\left(\log\left(\frac{n}{1+\frac{n}{m}+\frac1{s}}\right)-C_\varepsilon\right)$ for some constant $C_\varepsilon$ depending on $\varepsilon$.
\end{proof}

\begin{appendix}
\section{Proof of Proposition~\ref{P:bbprop}}\label{S:bbprop}
\begin{proof}[Proof of Proposition~\ref{P:bbprop}]Property~\ref{a:state} is immediate. Recall the process has equilibrium distribution
\[
\pi^{\mathrm{BB}(G,s,m)}(\xi)\propto \prod_{v\in V}\frac{\Gamma(s+\xi(v))}{\xi(v)!},\qquad\xi\in\Omega_{G,m}.
\]
Since $s=b/a$ with $a$ and $b$ coprime, this is of the form~\eqref{e:equil}:
\begin{align*}
\pi^{\mathrm{BB}(G,s,m)}(\xi)&\propto \prod_{\substack{v\in V:\\\xi(v)>0}}\frac1{\xi(v)!}\prod_{i=0}^{\xi(v)-1}(i+s)\\&=\prod_{\substack{\xi\in V:\\\xi(v)>0}}\frac1{\xi(v)!}\prod_{i=0}^{\xi(v)-1}\left(i+\frac{b}{a}\right)\propto\prod_{\substack{v\in V:\\\xi(v)>0}}\frac1{\xi(v)!}\prod_{i=0}^{\xi(v)-1}\left(ai+b\right).
\end{align*}Thus Property~\ref{a:eqm} holds. Property~\ref{a:sym} holds as a beta-binomial $X$ with parameters $(k,s,s)$ has the same distribution as $k-X$, for any $k\in\mathds{N}$.

To show Property~\ref{a:prob} holds (with $p^*=(5/12)^{2s}/(6B(s,s))$), we first show that with positive probability, if $\xi(v)+\xi(w)\ge2$ then $X/(\xi(v)+\xi(w))\in[1/3,2/3]$ where $X\sim\mathrm{BetaBin}(\xi(v)+\xi(w),s,s)$. Recall that to sample a BetaBin$(\xi(v)+\xi(w), s, s)$, we can first sample $Y\sim\mathrm{Beta}(s,s)$ and then given $Y$, sample Bin($\xi(v)+\xi(w),Y)$. We first observe that if $s\ge 1$, for such random variable $Y$, with probability at least $(5/12)^{2s}/(2B(s,s))$ (where $B(s,t)$ is the beta function), $Y$ will be in the interval $[5/12,7/12]$. This can be seen by noting that the density function of $Y$ in the interval $[5/12,7/12]$ is minimised on the boundary. If instead $s<1$, then $Y$ will be in $[5/12,7/12]$ with probability at least $(1/2)^{2s}/(2B(s,s))$. Further, if $s\ge20$ then we can use Chebyshev's inequality to obtain $\mathds{P}(Y\in[\frac{5}{12},\frac{7}{12}])\ge 1-\frac{36}{2s+1}\ge \frac12(1-\frac{20}{s+1})$.

Fix $y\in[5/12,7/12]$ and let $Z\sim \mathrm{Bin}(\xi(v)+\xi(w),y)$. We observe that $\mathds{P}(Z\in[(\xi(v)+\xi(w))/3,2(\xi(v)+\xi(w))/3])$ is minimized (over $\xi(v)+\xi(w)\ge 2$ and $y\in[5/12,7/12])$ when $\xi(v)+\xi(w)=4$ and $y=7/12$, with a value of $1225/3456>1/3$. Combining, we obtain that we can take $p^*=(5/12)^{2s}/(6B(s,s))$, and when $s\ge20$ we can take $p^*=\frac16(1-\frac{20}{s+1})$.

For the second part of Property~\ref{a:prob}, if $\xi(v)+\xi(w)=2$ then the probability that both particles end up on the same vertex is $1-\frac{s}{1+2s}$, which is larger than $2p^*$ for our choice of $p^*$.

For Property~\ref{a:MaBBindep}, observe that \begin{align*}\mathrm{P}_e^{\mathrm{BB}(G,s,m+1)}(C_{\xi,v},C_{\xi',v})&=\binom{\xi(v)+\xi(w)+1}{\xi'(v)+1}\frac{B(\xi'(v)+1+s,\xi'(w)+s)}{B(s,s)},\\
\mathrm{P}_e^{\mathrm{BB}(G,s,m+1)}(C_{\xi,v},C_{\xi',w})&=\binom{\xi(v)+\xi(w)+1}{\xi'(v)}\frac{B(\xi'(v)+s,\xi'(w)+1+s)}{B(s,s)},\\
\mathrm{P}_e^{\mathrm{BB}(G,s,m)}(\xi,\xi')&=\binom{\xi(v)+\xi(w)}{\xi'(v)}\frac{B(\xi'(v)+s,\xi'(w)+s)}{B(s,s)}.\end{align*}
Thus 
\begin{align*}
(\xi'(v)+1)\mathrm{P}_e^{\mathrm{BB}(G,s,m+1)}(C_{\xi,v},C_{\xi',v})&=\frac{(\xi(v)+\xi(w)+1)!}{\xi'(w)!\xi'(v)!}\frac{B(\xi'(v)+1+s,\xi'(w)+s)}{B(s,s)},\\
(\xi'(w)+1)\mathrm{P}_e^{\mathrm{BB}(G,s,m+1)}(C_{\xi,v},C_{\xi',w})&=\frac{(\xi(v)+\xi(w)+1)!}{\xi'(w)!\xi'(v)!}\frac{B(\xi'(v)+s,\xi'(w)+1+s)}{B(s,s)}.
\end{align*}
Adding these and using that $B(x,y)=B(x+1,y)+B(x,y+1)$, we obtain
\begin{align*}
&(\xi'(v)+1)\mathrm{P}_e^{\mathrm{BB}(G,s,m+1)}(C_{\xi,v},C_{\xi',v})+(\xi'(w)+1)\mathrm{P}_e^{\mathrm{BB}(G,s,m+1)}(C_{\xi,v},C_{\xi',w})\\&=\frac{(\xi(v)+\xi(w)+1)!}{\xi'(w)!\xi'(v)!}\frac{B(\xi'(v)+s,\xi'(w)+s)}{B(s,s)}\\
&=(\xi(v)+\xi(w)+1)\mathrm{P}_e^{\mathrm{BB}(G,s,m)}(\xi,\xi').\qedhere
\end{align*}
\end{proof}
\section{Proofs of Lemmas~\ref{L:thetaexist}--\ref{L:pairres}}
\label{S:App}

For ease of notation in this section we write $\mathrm{P}(v,v)$ for $\mathrm{P}_{e,B,B'}(v,v)$ and similarly for other probabilities. We shall use throughout (sometimes without reference) that, by Lemma~\ref{L:colour},
\begin{align}
\chi(B(v))\mathrm{P}_{}(v,v)+\chi(B(w))\mathrm{P}_{}(w,v)&=\chi(B'(v)),\label{e:fB'v}\\
\chi(B(v))\mathrm{P}_{}(v,w)+\chi(B(w))\mathrm{P}_{}(w,w)&=\chi(B'(w)).\label{e:fB'w}
\end{align}
We write $R_{v,w}$ for $R(v)+R(w)$, $P_{v,w}$ for $P(v)+P(w)$ and $B_{v,w}$ for $B(v)+B(w)$.
\begin{proof}[Proof of Lemma~\ref{L:thetaexist}]
We first show $m^*(v)\le u(v)+\frac12 u^P(v)$. Recall 
\[
u(v)+\frac12 u^P(v)=\chi(B'(v))\wedge R_{v,w}+\frac12\left(\left\{\chi(B'(v))-\chi(B'(v))\wedge R_{v,w}\right\}\wedge P_{v,w}\right).
\]
Hence if $R_{v,w}>\chi(B'(v))$ then $u(v)+\frac12 u^P(v)=\chi(B'(v))$. On the other hand in this case (using~\eqref{e:fB'v}), \[m^*(v)=\chi(B'(v))-[\chi(B(v)-R(v)-\frac12 P(v)]\mathrm{P}_{}(v,v)-[\chi(B(w)-R(w)-\frac12 P(w)]\mathrm{P}_{}(w,v)\]
and thus as $R(v)+P(v)\le \chi(B(v))$ and $R(w)+P(w)\le \chi(B(w))$, we have $m^*(v)\le \chi(B'(v))$, i.e.\! in this case $m^*(v)\le u(v)+\frac12 u^P(v)$. 

If instead $R_{v,w}\le \chi(B'(v))$, then $u(v)+\frac12 u^P(v)=R_{v,w}+\frac12(\{\chi(B'(v))-R_{v,w}\}\wedge P_{v,w})$. If $P_{v,w}>\chi(B'(v))-R_{v,w}$ then $u(v)+\frac12 u^P(v)=R_{v,w}+\frac12(\chi(B'(v))-R_{v,w})=\frac12(\chi(B'(v))+R_{v,w})$. But \begin{align*}m^*(v)&=\frac12 R(v)\mathrm{P}_{}(v,v)+\frac12 R(w)\mathrm{P}_{}(w,v)\\&\phantom{=}+\frac12\left\{(R(v)+P(v))\mathrm{P}_{}(v,v)+(R(w)+P(w))\mathrm{P}_{}(w,v)\right\}\\
&\le \frac12 R_{v,w}+\frac12\chi(B(v))\mathrm{P}_{}(v,v)+\frac12 \chi(B(w))\mathrm{P}_{}(w,v)\\
&=\frac12(\chi(B'(v))+R_{v,w}),
\end{align*}
using \eqref{e:fB'v} in the last step. If instead $P_{v,w}\le \chi(B'(v))-R_{v,w}$ then $u(v)+\frac12 u^P(v)=R_{v,w}+\frac12P_{v,w}$ and it is clear that this is an upper bound on $m^*(v)$ by bounding $\mathrm{P}_{}(v,v)$ and $\mathrm{P}_{}(w,v)$ by 1.

Now we turn to the lower bound, i.e.\! we want $m^*(v)\ge \ell(v)+\frac12\ell^P(v)$. Recall
\[
\ell(v)+\frac12\ell^P(v)=\{R_{v,w}-\chi(B'(w))\}\vee0+\frac12\left[\left\{P_{v,w}-\chi(B'(w))+R\wedge \chi(B'(w))\right\}\vee 0\right].
\]
If $R_{v,w}>\chi(B'(w))$, then $\ell(v)+\frac12\ell^P(v)=R_{v,w}-\chi(B'(w))+\frac12P_{v,w}.$ But in this case
\begin{align*}
m^*(v)&=(R(v)+\frac12 P(v))\mathrm{P}_{}(v,v)+(R(w)+\frac12P(w))\mathrm{P}_{}(w,v)\\
&=R_{v,w}+\frac12P_{v,w}-(R(v)+\frac12P(v))\mathrm{P}_{}(v,w)-(R(w)+\frac12P(w))\mathrm{P}_{}(w,w)\\
&=R_{v,w}+\frac12P_{v,w}-\chi(B'(w))+(\chi(B(v))-R(v)-\frac12P(v))\mathrm{P}_{}(v,w)\\&\phantom{=}+(\chi(B(w))-R(w)-\frac12P(w))\mathrm{P}_{}(w,w)\\
&\ge R_{v,w}+\frac12P_{v,w}-\chi(B'(w)).
\end{align*}
Finally suppose $R\le \chi(B'(v))$, then $\ell(v)+\ell^P(v)=\frac12\left[\left\{P_{v,w}-\chi(B'(w))+R_{v,w}\right\}\vee0\right]$. If further $P_{v,w}>\chi(B'(w))-R_{v,w}$ then $\ell(v)+\ell^P(v)=\frac12(P_{v,w}+R_{v,w}-\chi(B'(w))).$ But in this case 
\begin{align*}
m^*(v)&\ge \frac12\left[(R(v)+P(v))\mathrm{P}(v,v)+(R(w)+P(w))\mathrm{P}(w,v)\right]\\
&=\frac12\big[R_{v,w}+P_{v,w}-\chi(B'(w))+(\chi(B(v))-R(v)-P(v))\mathrm{P}(v,w)\\&\phantom{=\frac12\big[}+(\chi(B(w))-R(w)-P(w))\mathrm{P}(w,w)\big]\\
&\ge \frac12[R_{v,w}+P_{v,w}-\chi(B'(w))].
\end{align*}
If instead $P_{v,w}\le \chi(B'(w))-R_{v,w}$, then $\ell(v)+\ell^P(v)=0\le m^*(v)$.
\end{proof}
\begin{proof}[Proof of Lemma~\ref{L:thetabounds}]
Recall that we suppose $\mathrm{P}(v,v),\,\mathrm{P}(w,v)\in[\eta,1-\eta]$ and that $\theta(v)$ is defined in \eqref{e:thetadef} which gives
\[
\theta(v)=\frac{u(v)+\frac12 u^P(v)-m^*(v)}{u(v)+\frac12 u^P(v)-\ell(v)-\frac12\ell^P(v)}.
\]
There are numerous cases to consider which we detail below. Our goal is to show that in each case $\theta(v)\in[\eta,1-\eta]$. Recall that $\chi(B(v))+\chi(B(w))=\chi(B'(v))+\chi(B'(w))=aB_{v,w}+2b$.

\underline{Case 1: $R_{v,w}>\chi(B'(v))\vee \chi(B'(w))$}\\
In this case $u(v)+\frac12u^P(v)=\chi(B'(v))$ and \[u(v)+\frac12 u^P(v)-\ell(v)-\frac12\ell^P(v)=\chi(B'(v))-(R_{v,w}-\chi(B'(w)))-\frac12P_{v,w}.\]
But \begin{align*}m^*(v)&=\chi(B'(v))-(\chi(B(v))-R(v)-\frac12P(v))\mathrm{P}(v,v)-(\chi(B(w))-R(w)-\frac12P(w))\mathrm{P}(w,v)\\
&\le \chi(B'(v))-\eta(aB_{v,w}+2b-R_{v,w}-\frac12P_{v,w})\\
&=u+\frac12u^P(v)-\eta(u(v)+\frac12 u^P(v)-\ell(v)-\frac12\ell^P(v)),
\end{align*}
thus $\theta(v)\ge\eta$.

On the other hand
\begin{align*}
m^*(v)&=R_{v,w}+\frac12P_{v,w}-(R(v)+\frac12P(v))\mathrm{P}(v,w)-(R(w)+\frac12P(w))\mathrm{P}(w,w)\\
&=R_{v,w}+\frac12P_{v,w}-\chi(B'(w))+(\chi(B(v))-R(v)-\frac12P(v))\mathrm{P}(v,w)\\
&\phantom{=}+(\chi(B(w))-R(w)-\frac12P(w))\mathrm{P}(w,w)\\
&\ge R_{v,w}+\frac12P_{v,w}-\chi(B'(w))+\eta(aB+2b-R_{v,w}-\frac12P_{v,w})\\
&=\ell(v)+\ell^P(v)+\eta(u(v)+\frac12 u^P(v)-\ell(v)-\frac12\ell^P(v)),
\end{align*}
thus $1-\theta(v)\ge\eta$.

\underline{Case 2: $R_{v,w}\le \chi(B'(v))\wedge \chi(B'(w))$}\\
We consider sub-cases.

 \underline{Case 2a: $P_{v,w}\le (\chi(B'(v))-R_{v,w})\wedge(\chi(B'(w))-R_{v,w})$}\\
In this case $u(v)+\frac12u^P(v)=R_{v,w}+\frac12 P_{v,w}$ and $\ell(v)+\frac12\ell^P(v)=0$. But $m^*(v)\ge\eta(R_{v,w}+\frac12 P_{v,w})$ and so $\theta(v)\le 1-\eta$. We also have $m^*(v)\le (1-\eta)(R_{v,w}+\frac12 P_{v,w})$ and so $\theta\ge \eta$.

\underline{Case 2b: $P_{v,w}\ge (\chi(B'(v))-R_{v,w})\vee(\chi(B'(w))-R_{v,w})$}\\
In this case $u(v)+\frac12u^P(v)=\frac12(\chi(B'(v))+R_{v,w})$ and $\ell(v)+\frac12\ell^P(v)=\frac12(P_{v,w}+R_{v,w}-\chi(B'(w)))$. Hence
\[
u(v)+\frac12u^P(v)-\ell(v)-\frac12\ell^P(v)=\frac12(aB_{v,w}+2b-P_{v,w}).
\]On the one hand
\begin{align}\notag
m^*(v)&=\frac12 R(v)\mathrm{P}(v,v)+\frac12 R(w)\mathrm{P}(w,v)+\frac12\left\{(R(v)+P(v))\mathrm{P}(v,v)+(R(w)+P(w))\mathrm{P}(w,v)\right\}\\\notag
&\ge \frac12\eta R_{v,w}+\frac12\left\{R_{v,w}+P_{v,w}-(R(v)+P(v))\mathrm{P}(v,w)-(R(w)+P(w))\mathrm{P}(w,w)\right\}\\\notag
&=\frac12\eta R_{v,w}+\frac12\big\{R_{v,w}+P_{v,w}-\chi(B'(w))+(\chi(B(v))-R(v)-P(v))\mathrm{P}(v,w)\\\notag
&\phantom{=\frac12\eta R_{v,w}+\frac12\big\{}+(\chi(B(w))-R(w)-P(w))\mathrm{P}(w,w)\big\}\\
&\ge \frac12\eta R_{v,w}+\frac12(R_{v,w}+P_{v,w}-\chi(B'(w)))+\frac12\eta(aB_{v,w}+2b-R_{v,w}-P_{v,w}).\label{e:m2b1}
\end{align}
Thus $m^*(v)-\ell(v)-\frac12\ell^P(v)\ge \frac12\eta(aB_{v,w}+2b-P_{v,w})$, and so $1-\theta(v)\ge \eta$.
On the other hand,
\begin{align*}
m^*(v)&=\frac12 R(v)\mathrm{P}(v,v)+\frac12 R(w)\mathrm{P}(w,v)+\frac12 \chi(B'(v))\\
&\phantom{=}-\frac12\big\{(\chi(B(v))-R(v)-P(v))\mathrm{P}(v,v)+(\chi(B(w))-R(w)-P(w))\mathrm{P}(w,v)\big\}\\
&\le \frac12\eta R_{v,w}+\frac12 \chi(B'(v))-\frac12\eta(aB_{v,w}+2b-R_{v,w}-P_{v,w})
\end{align*}
Thus \begin{align}\notag
u(v)+\frac12 u^P(v)-m^*(v)&\ge\frac12(1-\eta)R_{v,w}+\frac12\eta(aB_{v,w}+2b-R_{v,w}-P_{v,w})\\
&\ge \frac12\eta(aB_{v,w}+2b-P_{v,w})\label{e:m2b}
\end{align}
where we use $\eta<1/2$ in the last inequality. This gives $\theta(v)\ge \eta$.

\underline{Case 2c: $\chi(B'(v))-R_{v,w}\le P_{v,w}\le \chi(B'(w))-R_{v,w}$}\\
We have $u(v)+\frac12 u^P(v)=\frac12(\chi(B'(v))+R_{v,w})$ and $\ell(v)+\frac12\ell^P(v)=0$.  On the one hand $m^*(v)\ge \eta(R_{v,w}+\frac12 P_{v,w})=\frac12\eta R_{v,w}+\frac12\eta(R_{v,w}+P_{v,w})\ge \frac12\eta R_{v,w}+\frac12\eta \chi(B'(v))=\frac12\eta(R_{v,w}+\chi(B'(v)))$. Hence $\theta(v)\le 1-\eta$. On the other hand, as in~\eqref{e:m2b} we have $u(v)+\frac12 u^P(v)-m^*(v)\ge\frac12\eta(aB_{v,w}+2b-P_{v,w})$ which in this case becomes $u(v)+\frac12 u^P(v)-m^*(v)\ge\frac12\eta(\chi(B'(v))+R_{v,w})$. This gives $\theta(v)\ge \eta$.

\underline{Case 2d: $\chi(B'(w))-R_{v,w}\le P_{v,w}\le \chi(B'(v))-R_{v,w}$}\\
We have $u(v)+\frac12 u^P(v)=R_{v,w}+\frac12 P_{v,w}$ and $\ell(v)+\frac12\ell^P(v)=\frac12(R_{v,w}+P_{v,w}-\chi(B'(w)))$. As in~\eqref{e:m2b}, $u(v)+\frac12 u^P(v)-m^*(v)\ge\frac12\eta(aB_{v,w}+2b-P_{v,w})$ which gives $u(v)+\frac12 u^P(v)-m^*(v)\ge \frac12\eta(R_{v,w}+\chi(B'(w)))$, so $1-\theta(v)\ge\eta$. On the other hand, $m^*(v)\le (1-\eta)(R_{v,w}+\frac12 P_{v,w})$, but $R_{v,w}+\frac12 P_{v,w}\ge R_{v,w}+\frac12(\chi(B'(w))-\frac12 R_{v,w})=\frac12 (R_{v,w}+\chi(B'(w)))$ and so $m^*(v)\le R_{v,w}+\frac12 P_{v,w}-\frac{\eta}{2}(R_{v,w}+\chi(B'(w)))$ which gives $\theta(v)\ge \eta$.

\underline{Case 3: $\chi(B'(v))\le R_{v,w}\le \chi(B'(w))$}\\ 
In this case $u(v)+\frac12 u^P(v)=\chi(B'(v))$. We again consider sub-cases depending on the value of $P_{v,w}$.

\underline{Case 3a: $P_{v,w}\ge \chi(B'(w))-R_{v,w}$}\\
Then $\ell(v)+\frac12\ell^P(v)=\frac12(R_{v,w}+P_{v,w}-\chi(B'(w)))$ and so $u(v)+\frac12 u^P(v)-\ell(v)-\frac12\ell^P(v)=\chi(B'(v))-\frac12(R_{v,w}+P_{v,w}-\chi(B'(w)))$.
To show $1-\theta(v)\ge \eta$, we wish to show that $m^*(v)\ge \ell(v)+\frac12\ell^P(v)+\eta(u(v)+\frac12 u^P(v)-\ell(v)-\frac12\ell^P(v))$, i.e.\! that $m^*(v)\ge \eta \chi(B'(v))+\frac12(1-\eta)(R_{v,w}+P_{v,w}-\chi(B'(w)))$. As in~\eqref{e:m2b1} we have
\begin{align*}
m^*(v)&\ge \frac12 \eta R_{v,w}+\frac12\left\{R_{v,w}+P_{v,w}-\chi(B'(w))+\eta(aB_{v,w}+2b-R_{v,w}-P_{v,w})\right\}\\
&=\frac12(1-\eta)(R_{v,w}+P_{v,w}-\chi(B'(w)))+\frac12\eta\left\{R_{v,w}+P_{v,w}-\chi(B'(w))+aB_{v,w}+2b-P_{v,w}\right\}.
\end{align*}
But $R_{v,w}-\chi(B'(w))+aB_{v,w}+2b=R_{v,w}+\chi(B'(v))\ge 2\chi(B'(v))$, so  $m^*(v)\ge \eta \chi(B'(v))+\frac12(1-\eta)(R_{v,w}+P_{v,w}-\chi(B'(w)))$ as needed. On the other hand, to show $\theta(v)\ge\eta$, we need to show that $m^*(v)\le (1-\eta)\chi(B'(v))+\frac{\eta}{2}(P_{v,w}+R_{v,w}-\chi(B'(w)))$. We have
\begin{align*}
&m^*(v)\\&=(1-\eta)\big\{(R(v)+\frac12P(v))\mathrm{P}(v,v)+(R(w)+\frac12P(w))\mathrm{P}(w,v)\big\}\\
&\phantom{=}+\eta\big\{(R(v)+\frac12 P(v))\mathrm{P}(v,v)+(R(w)+\frac12P(w))\mathrm{P}(w,v)\big\}\\
&=(1-\eta)\big\{\chi(B'(v))-(\chi(B(v))-R(v)-\frac12 P(v))\mathrm{P}(v,v)-(\chi(B(w))-R(w)-\frac12P(w))\mathrm{P}(w,v))\big\}\\
&\phantom{=}+\eta\big\{R_{v,w}+\frac12P_{v,w}-(R(v)+\frac12P(v))\mathrm{P}(w,v)-(R(w)+\frac12 P(w))\mathrm{P}(w,w)\big\}\\
&\le (1-\eta)\chi(B'(v))\\&\phantom{=}+\eta\big\{R_{v,w}+\frac12P_{v,w}-\chi(B'(w))-(\chi(B(v))-R(v)-\frac12P(v))\mathrm{P}(w,v)\\&\phantom{\,=+\eta\big\{R_{v,w}+\frac12P_{v,w}-\chi(B'(w))}-(\chi(B(w))-R(w)-\frac12P(w))\mathrm{P}(w,w)\big\}\\
&\le(1-\eta)\chi(B'(v))+\eta(R_{v,w}+\frac12P_{v,w}-\chi(B'(w)))\\
&\le (1-\eta)\chi(B'(v))+\frac{\eta}{2}(R_{v,w}+P_{v,w}-\chi(B'(w))),
\end{align*}
as needed.

\underline{Case 3b: $P_{v,w}\le \chi(B'(w))-R_{v,w}$}\\
Here $u(v)+\frac12u^P(v)=\chi(B'(v))$ and $\ell(v)+\frac12\ell^P(v)=0$. On the one hand we have $m^*(v)\ge \eta(R_{v,w}+\frac12P_{v,w})\ge\eta R_{v,w}\ge\eta \chi(B'(v))$. This gives $\theta(v)\le 1-\eta$. On the other hand, we have 
\begin{align*}
m^*(v)&=\chi(B'(v))-(\chi(B(v))-R(v)-P(v))\mathrm{P}(v,v)-\frac12 P(v)\mathrm{P}(v,v)\\
&\phantom{\le \chi(B'(v))\,\,}-(\chi(B(w))-R(w)-P(w))\mathrm{P}(w,v)-\frac12P(w)\mathrm{P}(w,v)\\
&\le \chi(B'(v))-\eta(aB_{v,w}+2b-R_{v,w}-\frac12 P_{v,w})-\frac12\eta P_{v,w}\\
&\le \chi(B'(v))-\eta \chi(B'(v))\\
&=(1-\eta)\chi(B'(v)).
\end{align*}Thus it follows that $\theta(v)\ge \eta$.
 
\underline{Case 4: $\chi(B'(w))\le R_{v,w}\le \chi(B'(v))$}\\
This is the final main case, and it has two sub-cases.

\underline{Case 4a: $P_{v,w}\ge \chi(B'(v))-R_{v,w}$}\\
Here $u(v)+\frac12u^P(v)=\frac12(\chi(B'(v))+R_{v,w})$ and $\ell(v)+\frac12\ell^P(v)=R_{v,w}-\chi(B'(w))+\frac12P_{v,w}$. Thus $u(v)+\frac12u^P(v)-\ell(v)-\frac12\ell^P(v)=\chi(B'(w))+\frac12(\chi(B'(v))-R_{v,w}-P_{v,w})$.

On the one hand we want $\theta(v)\le 1-\eta$, which in this case is equivalent to $m^*(v)\ge (1-\frac{\eta}{2})R_{v,w}+\frac{\eta}{2}\chi(B'(v))+\frac{1-\eta}{2}P_{v,w}-(1-\eta)\chi(B'(w))$. We can obtain this bound since
\begin{align*}
m^*(v)&=R_{v,w}+\frac12P_{v,w}-\chi(B'(w))+(\chi(B(v))-R(v)-\frac12P(v))\mathrm{P}(v,w)\\
&\phantom{=}+(\chi(B(w))-R(w)-\frac12P(w))\mathrm{P}(w,w)\\
&\ge R_{v,w}+\frac12P_{v,w}-\chi(B'(w))+\eta(aB_{v,w}+2b-R_{v,w}-\frac12P_{v,w})\\
&=R_{v,w}+\frac12P_{v,w}-\chi(B'(w))+\eta(\chi(B'(v))+\chi(B'(w))-R_{v,w}-\frac12P_{v,w})\\
&=(1-\frac{\eta}{2})R_{v,w}+\frac{1-\eta}{2}P_{v,w}+\frac{\eta}{2}\chi(B'(v))-(1-\eta)\chi(B'(w))+\frac{\eta}{2}(\chi(B'(v))-R_{v,w}),
\end{align*}
and we obtain the desired bound using that $R_{v,w}\le \chi(B'(v))$.

On the other hand, we also need to show $\theta(v)\le \eta$, i.e.\! we need to show $m^*(v)\le\frac{1+\eta}{2}R_{v,w}+\frac{\eta}{2}P_{v,w}+\frac{1-\eta}{2}\chi(B'(v))-\eta \chi(B'(w))$. We have 
\begin{align*}
m^*(v)&=\frac12R(v)\mathrm{P}(v,v)+\frac12R(w)\mathrm{P}(w,v)+\frac12\left\{(R(v)+P(v))\mathrm{P}(v,v)+(R(w)+P(w))\mathrm{P}(w,v)\right\}\\
&\le \frac{1-\eta}{2}R_{v,w}+\frac12\big\{\chi(B'(v))-(\chi(B(v))-R(v)-P(v))\mathrm{P}(v,v)\\
&\phantom{\le\frac{1-\eta}{2}R_{v,w}+\frac12\big\{}-(\chi(B(w))-R(w)-P(w))\mathrm{P}(w,v)\big\}\\
&\le \frac{1-\eta}{2}R_{v,w}+\frac12\big\{\chi(B'(v))-\eta(aB_{v,w}+2b-R_{v,w}-P_{v,w})\big\}\\
&=\frac{1-\eta}{2}R_{v,w}+\frac{\eta}{2}P_{v,w}+\frac{1-\eta}{2}\chi(B'(v))-\frac{\eta}{2}\chi(B'(w))\\
&=\frac{1+\eta}{2}R_{v,w}+\frac{\eta}{2}P_{v,w}+\frac{1-\eta}{2}\chi(B'(v))-\eta \chi(B'(w))-\eta R_{v,w}+\frac{\eta}{2}\chi(B'(w))\\
&\le\frac{1+\eta}{2}R_{v,w}+\frac{\eta}{2}P_{v,w}+\frac{1-\eta}{2}\chi(B'(v))-\eta \chi(B'(w)),
\end{align*}
using that $R_{v,w}\ge \chi(B'(w))$ in the last inequality.

\underline{Case 4b: $P_{v,w}\le \chi(B'(v))-R_{v,w}$}\\
Here $u(v)+\frac12u^P(v)=R_{v,w}+\frac12 P_{v,w}$ and $\ell(v)+\frac12 \ell^P(v)=R_{v,w}-\chi(B'(w))+\frac12 P_{v,w}$, thus $u(v)+\frac12u^P(v)-\ell(v)-\frac12 \ell^P(v)=\chi(B'(w))$. Showing $\theta(v)\le 1-\eta$ is equivalent to showing $m^*(v)\ge R_{v,w}+\frac12 P_{v,w}-(1-\eta)\chi(B'(w))$. We have 
\begin{align*}
m^*(v)&=\chi(B'(v))-(\chi(B(v))-R(v)-P(v))\mathrm{P}(v,v)-\frac12 P(v)\mathrm{P}(v,v)\\
&\phantom{\le \chi(B'(v))\,\,}-(\chi(B(w))-R(w)-P(w))\mathrm{P}(w,v)-\frac12P(w)\mathrm{P}(w,v)\\
&\ge \chi(B'(v))-(1-\eta)(aB_{v,w}+2b-R_{v,w}-P_{v,w})-\frac{1-\eta}{2}P_{v,w}\\
&=R_{v,w}+\frac12P_{v,w}-(1-\eta)\chi(B'(w))+\eta(\chi(B'(v))-R_{v,w}),
\end{align*}and this shows the desired bound since $R_{v,w}\le \chi(B'(v))$. Showing $\theta(v)\ge\eta$ is equivalent to showing $m^*(v)\le R_{v,w}+\frac12 P_{v,w}-\eta \chi(B'(w))$. This holds since we have $m^*(v)\le (1-\eta)(R_{v,w}+\frac12P_{v,w})\le R_{v,w}+\frac12 P_{v,w}-\eta R_{v,w}$ and since $R_{v,w}\ge \chi(B'(w))$ this gives $m^*(v)\le R_{v,w}+\frac12 P_{v,w}-\eta \chi(B'(w))$ as needed.
\end{proof}
\begin{proof}[Proof of Lemma~\ref{L:pairres}]
We fix the configurations of black, red, and pink particles $B$, $R$, $P$ just before an update on $e=\{v,w\}$ and also the number of paired red $R^p_{v,w}$. Let $x=\ell(v)+\ell(w)-R^q_{v,w}$, where $R^q_{v,w}$ is the number of non-paired red particles on $e$. Then $x\vee0$ is the number of paired red particles needed for the lower bounds in Step 1 and so any particular paired red particle will be remaining in the pool after Step 1 with probability $1-(x\vee 0)/R^p_{v,w}$. Observe that $\chi(B'(v))+\chi(B'(w))\ge R_{v,w}+R^p_{v,w}$ (since each paired red particle on $e$ implies the existence of a unique paired white particle also on~$e$). We consider four cases. 

\underline{Case 1: $R_{v,w}>\chi(B'(v))\vee \chi(B'(w))$}\\
Then $x=2R_{v,w}-\chi(B'(v))-\chi(B'(w))-R^q_{v,w}\le 2R_{v,w}-(R_{v,w}+R^p_{v,w})-R^q_{v,w}=0,$ i.e.\! no paired red particles are needed for the lower bounds and they all remain in the pool after Step 1.

\underline{Case 2: $R_{v,w}\le \chi(B'(v))\wedge \chi(B'(w))$}\\
In this case $x=-R^q_{v,w}$ so all paired red particles remain in the pool.

\underline{Case 3: $\chi(B'(v))\le R_{v,w}\le \chi(B'(w))$}\\
Then $x=R_{v,w}-\chi(B'(v))-R^q_{v,w}=R^p_{v,w}-\chi(B'(v))$. We need to show that this is at most $(1-\gamma) R^p_{v,w}$. We are assuming that $\chi(B'(w))\le \chi(B'(v))/\gamma$. We also have that $\chi(B'(v))+\chi(B'(w))\ge 2R^p_{v,w}$ and thus $\chi(B'(v))\ge 2R^p_{v,w}/(1+1/\gamma)\ge \gamma R^p_{v,w}$ since $\gamma<1$. This gives the desired bound on $x$.

\underline{Case 4: $\chi(B'(w))\le R_{v,w}\le \chi(B'(v))$}\\
This case follows similarly to Case 3, switching the roles of $v$ and $w$.
\end{proof}

\section{Simulation}\label{A:sim}
For purposes of further elucidating the evolution of the chameleon process and its relationship to the MaBB, we present a possible trajectory of the two processes for two updates (for simplicity we suppose the first two edge-rings occur at times 1 and 2). In this example, the graph is the line on 7 vertices and $a=b=1$.

\begin{figure}[!h]\begin{center}\includegraphics[scale=1.2]{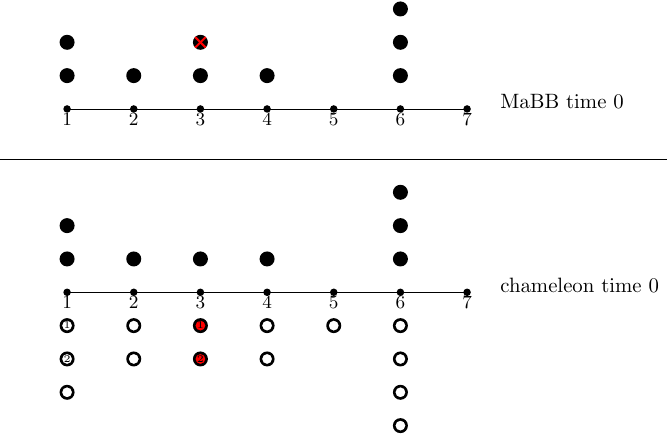}\caption{The initial configurations are shown as above. Observe that the non-marked particles in the MaBB are in the same configuration as the black particles in the chameleon and vertex 3 (which has the marked particle in the MaBB) has all its non-black particles in the chameleon configuration as red. As it is the start of a round, and as there are fewer red particles than white, we pair up each red particle with a unique white particle and label the paired particles (with the same label) to track the pairings.}\end{center}\end{figure}
\begin{figure}[!h]\begin{center}
\includegraphics[scale=1.2]{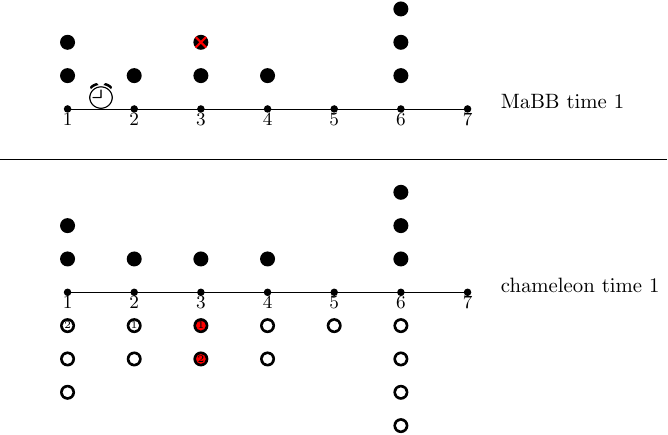}\caption{At time 1 edge $\{1,2\}$ rings and although this does not lead to a change in the non-marked particles, one of the labelled white particles moves as indicated.}\end{center}\end{figure}
\begin{figure}[!h]\begin{center}
\includegraphics[scale=1.2]{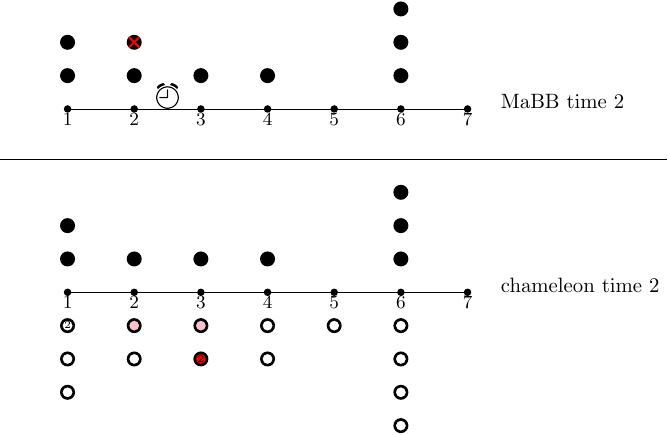}\caption{At time 2 edge $\{2,3\}$ rings and the marked particle ends up on vertex 2 in the MaBB. In the chameleon process there is a red-white pair of particles (with label 1) on the ringing edge, thus an opportunity for pink particles to be created which we see happen in this simulation. One pink particle is created on vertex 2 and the other on vertex 3. At the next depinking time these pink particles will either both become red or both become white.}\end{center}\end{figure}

\end{appendix}
\clearpage
\bibliography{reshuffling}

\end{document}